\documentclass{article}
\usepackage{amsmath, amsthm, amsfonts, amssymb, amscd, stmaryrd,xcolor,enumerate}

\usepackage{graphicx}
\usepackage{tikz}

\theoremstyle{definition}

\theoremstyle{plain}
\newtheorem{Thm}{Theorem}
\newtheorem{Lem}{Lemma}
\newtheorem{Cor}{Corollary}

\newtheorem{Prop}{Proposition}

\numberwithin{equation}{section}		
\numberwithin{figure}{Prob}
\numberwithin{table}{Prob}
\numberwithin{Thm}{section}
\numberwithin{Lem}{section}
\numberwithin{Def}{section}
\numberwithin{Prop}{section}
\numberwithin{Cor}{section}

\newcommand{\Z}{\mathbb{Z}}

\newcommand{\trace}{\text{tr}}

\title{Almost Commutative Terwilliger Algebras of Group Association Schemes II: Primitive Idempotents}
\author{Nicholas L. Bastian}

\date{\today}

\begin{document}

\maketitle

\begin{abstract}
This paper is a continuation of Almost Commutative Terwilliger Algebras of Group Association Schemes I: Classification \cite{Bastian1}. In that paper, we found all groups $G$ for which the Terwilliger algebra of the group association scheme, denoted $T(G)$, is almost commutative. We also found the primitive idempotents for $T(G)$ for three of the four types of such groups. In this paper, we determine the primitive idempotents for the fourth type. 
  \end{abstract}



\section{Introduction}

Camina $p-$groups $G$ have particularly nice properties. One property that is of particular interest is that $G'$ and $Z(G)$ are both elementary abelian. As a result of this structure the group association scheme for Camina $p-$groups is also particularly nice. 

In the first paper \cite{Bastian1}, we showed that given any finite group $G$, the Terwilliger algebra for the group association scheme, denoted $T(G)$, is almost commutative if and only if $G$ is abelian, a Frobenius group of the from $(\Z_p)^n\rtimes \Z_{p^n-1}$ for prime $p$ and $n>0$, the group $(\Z_3)^2\rtimes Q_8$ where $Q_8$ is the quaternion group of order $8$, or a Camina $p-$group for prime $p$. We also determined the dimension and all the primitive idempotents for $T(G)$ for the first three of these four types of groups. This paper focuses on finding the dimension and non-primary primitive idempotents for $T(G)$ when $G$ is a Camina $p-$group.  

This paper is organized as follows. In Section $2$ we shall discuss Camina groups as well as results that will be useful in determining the primitive idempotents for their Terwilliger algebra. Section $3$ will discuss what a Terwilliger algebra is as well as what it means for a Terwilliger algebra to be almost commutative. We here also mention the results of the previous paper \cite{Bastian1} that are relevant to this paper. In Section $4$ we shall determine the primitive idempotents for $T(G)$ for all Camina $p-$groups $G$.

\section{Camina Groups}

A group $G$ is a \emph{Camina group} \cite{Camina,Lewis1} if every conjugacy class of $G$ outside of $G'$ is a coset of $G'$ in $G$. In this paper our interest is in Camina $p-$groups. It has also been shown by Dark and Scoppola \cite{Dark1} that every Camina $p-$group has nilpotency class $2$ or $3$. We mention here results that will be of particular interest to us.

\begin{Prop}[Macdonald, \cite{MacDonald}]\label{thm:Camina3}
        Let $G$ be a Camina $p-$group of nilpotency class $3$. Then the lower central series of $G$ is $\{e\}< Z(G)< G'< G$. Furthermore, $[G\colon G']=p^{2n}$ and $[G'\colon Z(G)]=p^n$ for some even integer $n>0$. Additionally, $G/G'$ and $G'/Z(G)$ are elementary abelian groups.
\end{Prop}

\begin{Prop}[Macdonald, \cite{MacDonald}]\label{prop:cam2center}
        Let $G$ be a Camina p-group. Then $Z(G)$ is elementary abelian.
    \end{Prop}

     \begin{Lem}[\cite{Lewis1}]\label{lem:dercam}
            Let $G$ be a non-abelian Camina group. Then $Z(G)\leq G'$.
        \end{Lem}

        \begin{Prop}[\cite{Lewis1}]\label{prop:cammod}
        Suppose that $G$ is a Camina group and $N\trianglelefteq G'$. Then $G/N$ is a Camina group.
    \end{Prop}

\begin{Prop}[Macdonald, \cite{MacDonald}]\label{lem:pcamclass}
    Let $G$ be a Camina $p-$group. Let $\gamma_2(G)=[G,G]$ and $\gamma_3(G)=[\gamma_2(G),G]$ be the second and third terms in the lower central series of $G$. Then
        \[x^G=x\gamma_2(G)\text{ if }x\in G\setminus \gamma_2(G),\]
        \[x^G=x\gamma_3(G)\text{ if }x\in \gamma_2(G)\setminus \gamma_3(G),\]
        \[x^G=\{x\} \text{ if }x\in \gamma_3(G).\]
\end{Prop}

For $H\triangleleft G$, the pair $(G,H)$ is a \emph{Camina pair} if for all $g\in G\setminus H$, $g$ is conjugate to every element of $gH$. Note $(G,G')$ is a Camina pair if and only if $G$ is a Camina group. The following result was proven over the course of several papers \cite{Camina, Chillag1, Kuisch, Yongcai}. 

\begin{Lem}\label{lem:pair}
    Let $1< K\triangleleft G$. Then the following are equivalent:
    \begin{enumerate}
        \item $(G,K)$ is a Camina pair.
        \item If $x\in G\setminus K$, then $|C_G(x)|=|C_{G/K}(xK)|$.
        \item If $xK$ and $yK$ are conjugate and nontrivial in $G/K$, then $x$ is conjugate to $y$ in $G$.
        \item If $C_1=\{e\},C_2,\cdots , C_m$ are the conjugacy classes of $G$ contained in $K$ and $C_{m+1},\cdots, C_n$ are the conjugacy classes of $G$ outside $K$, then $C_iC_j=C_j$ for $1\leq i\leq m$ and $m+1\leq j\leq n$.
    \end{enumerate}
\end{Lem}

\section{Terwilliger Algebras}

Definitions of Terwilliger algebras and Schur rings can be found in \cite{Bastian1}. One particular definition we remind the reader of is that given a finite group $G$ with conjugacy classes $C_0,C_1,\cdots, C_d$, we can define $R_i=\{(x,y)\in G\times G\colon yx^{-1}\in C_i\}$. Then $\mathcal{A}=(G,\{R_i\}_{0\leq i\leq d})$ is an association scheme called the \emph{group association scheme}. We let $T(G)$ denote the Terwilliger algebra for the group association scheme of group $G$. For any $x\in G$, we define the primary-T(x) module in \cite{Bastian1}. This definition will be used in the next important definition.

Recall that for any $x\in G$, the Terwilliger algebra $T(x)$ is \emph{almost commutative} if every non-primary irreducible $T(x)-$module is $1-$ dimensional.

\begin{Thm}[Tanaka, \cite{Tanaka}]\label{thm:tanaka}  Let $\mathcal{A}=(\Omega,\{R_i\}_{0\leq i\leq d})$ be a commutative association scheme. Let $\mathcal{T}(x)$ be the Terwilliger algebra of $\mathcal{A}$ for some $x\in \Omega$. The following are equivalent:
    \begin{enumerate}[(i)]
        \item Every non-primary irreducible $\mathcal{T}(x)-$module is $1-$ dimensional for some $x\in \Omega$.
        \item Every non-primary irreducible $\mathcal{T}(x)-$module is $1-$ dimensional for all $x\in \Omega$.
        \item The intersection numbers of $\mathcal{A}$ have the property that for all distinct $h,i$ there is exactly one $j$ such that $p_{ij}^h\neq 0$ $(0\leq h,i,j\leq d)$.
        \item The Krein parameters of $\mathcal{A}$ have the property that for all distinct $h,i$ there is exactly one $j$ such that $q_{ij}^h\neq 0$ $(0\leq h,i,j\leq d)$.
        \item $\mathcal{A}$ is a wreath product of association schemes $\mathcal{A}_1,\mathcal{A}_2,\cdots, \mathcal{A}_n$ where each $\mathcal{A}_i$ is either a $1-$class association scheme or the group scheme of a finite abelian group.
    \end{enumerate} 
Moreover, a Terwilliger algebra satisfying these equivalent conditions is triply regular.
\end{Thm}

Definitions of intersection numbers and Krein parameters are given in \cite{Bastian1}. For the definition of the wreath product of association schemes see \cite{WreathProd}. We note that a Terwilliger algebra for a commutative association scheme is almost commutative if it satisfies the five equivalent conditions in Theorem \ref{thm:tanaka}.

In \cite{Bastian1}, we found those groups for which $T(G)$ is almost commutative. We mention this result here.

\begin{Thm}[\cite{Bastian1}]\label{thm:acclassify}
        Let $G$ be a finite group. Then $T(G)$ is almost commutative if and only if $G$ is isomorphic to one of the following groups
        \begin{itemize}
        \item A finite abelian group
        \item The group $(\Z_3)^2\rtimes Q_8$, where $Q_8$ is the quaternions of order $8$
        \item $(\Z_p)^r\rtimes \Z_{p^{r-1}}$ for some prime $p$ and $r>0$.
        \item A non-abelian Camina $p-$group, for some prime $p$.
    \end{itemize}
    \end{Thm}

\section{Wedderburn Decomposition for Camina p-groups}
In \cite{Bastian1}, we found the primitive idempotents for $T(G)$ for the first three of the four types of groups in Theorem \ref{thm:acclassify}. We shall now determine the primitive idempotents for $T(G)$ along with the Wedderburn decomposition for all Camina $p-$groups $G$. From this point on let $G$ be a Camina $p-$group. We also let $I_n$ denote the $n\times n$ identity matrix and $J_n$ denote the $n\times n$ all ones matrix. 

We shall consider two cases based on the nilpotency class of $G$, proving the results for both at the same time where possible. We start with the dimensions of the algebras.

 \begin{Prop}\label{cor:p2dim}
            Let $G$ have nilpotency class $2$, where $|G|=p^n$ and $|Z(G)|=p^k$. Then $\dim T(G)=(p^{n-k}+p^k-1)(p^{n-k}+p^k-2)+p^{n}$.
        \end{Prop}

\begin{proof}
By Theorem \ref{thm:acclassify}, $T(G)$ is almost commutative. Then by Theorem \ref{thm:tanaka}, $T(G)$ is triply regular. Therefore, $\dim T(G)=|\{(i,j,h)\colon p_{ij}^h\neq 0\}|$. As $G$ has nilpotency class $2$, its lower central series is $\{e\}\triangleleft G_1\triangleleft G$. We have $G_1\leq Z(G)$ and $G/G_1\leq Z(G/G_1)$. Note that $G/G_1$. Thus, $G'\leq G_1\leq Z(G)$. So $G'\leq Z(G)$. By Lemma \ref{lem:dercam}, $Z(G)\leq G'$. So $G'=Z(G)$.

    As $G$ is a Camina group the conjugacy classes of $G$ are those inside of $G'=Z(G)$ and the non-trivial cosets of $G/G'$. As $|G/G'|=p^{n-k}$ there are $p^{n-k}-1$ conjugacy classes that are non-trivial cosets of $G/G'$. Since $G'=Z(G)$ the classes of $G$ in $G'$ are just singleton elements. There are $p^k$ of these. Then $G$ has $p^{n-k}-1+p^k$ classes. We shall denote these classes as $C_i$. We consider cases for fixed $i,h$ to find the number of nonzero constants $p_{ij}^h$.

    {\bf Case 1: } $i\neq h$. As $T(G)$ is almost commutative, by Theorem \ref{thm:acclassify} there is a unique $j$ such that $p_{ij}^h\neq 0$. Thus there is a unique $j$ such that $p_{ij}^h\neq 0$. We have $p^{n-k}+p^k-1$ choices for $C_i$. We then have $p^{n-k}+p^k-2$ choices for $C_h$. Therefore, in total we have $(p^{n-k}+p^k-1)(p^{n-k}+p^k-2)$ triples $(i,j,h)$ with $p_{ij}^h\neq 0$ in this case.

    {\bf Case 2: }$i=h$ and $C_i=\{g\}\subseteq Z(G)$. For any $C_j\not\subseteq Z(G)$ we have $C_j=xZ(G)$, $x\not\in Z(G)$. Then $\overline{C_i}\cdot\overline{C_j}=\overline{gxZ(G)}=\overline{xZ(G)}$, which does not contain $g\in Z(G)$. Thus, $p_{ij}^h=0$ in this case. If $C_j=\{x\}\subseteq Z(G)$, then $\overline{C_i}\cdot\overline{C_j}=gx$. This equals $\overline{C_h}$ if and only if $x=e$. Thus, for $i=h$, there is a unique $C_j=\{e\}$ such that $p_{ij}^h\neq 0$. We have $p^k$ choices for $i=h$, so there are a total of $p^k$ triples $(i,j,h)$ with $p_{ij}^h\neq 0$ in this case.

    {\bf Case 3: }$i=h$ with $C_i\not\subseteq Z(G)$. Then $C_i=gZ(G)$, $g\not\in Z(G)$. For any $C_j\not\in Z(G)$ we have $C_j=xZ(G)$ for $x\not\in Z(G)$. Then $\overline{C_i}\cdot\overline{C_j}=|Z(G)|\overline{gxZ(G)}$. As $g\in gZ(G)$, for $p_{ij}^h\neq 0$ we must have $gxZ(G)=gZ(G)$. Then $xZ(G)=Z(G)$, which is false. Thus, $p_{ij}^h=0$ for all $C_j\not\subseteq Z(G)$. For any $C_j\subseteq Z(G)$ we have $C_j=\{x\}\subseteq Z(G)$. Then $\overline{C_i}\cdot\overline{C_j}=\overline{gxZ(G)}=\overline{gZ(G)}=\overline{C_h}$. So $p_{ij}^h=1$ in this case. Hence, for any $C_j\subseteq Z(G)$ we have $p_{ij}^h\neq 0$. We have $p^k$ choices for $C_j$ for each $C_i=C_h$. There are $p^{n-k}-1$ choices for $C_i$. Thus, in total we have $p^k(p^{n-k}-1)$ triples $(i,j,h)$ such that $p_{ij}^h\neq 0$ in this case.

    Having considered every possible case we obtain:
    \[\dim T(G)=(p^{n-k}+p^k-1)(p^{n-k}+p^k-2)+p^k+p^{k}(p^{n-k}-1)=(p^{n-k}+p^k-1)(p^{n-k}+p^k-2)+p^n.\qedhere\]
\end{proof}

\begin{Prop}\label{cor:p3dim}
    Let $G$ have nilpotency class $3$. Suppose that $|Z(G)|=p^k$, $[G\colon G']=p^{2n}$, and $[G'\colon Z(G)]=p^n$. Then 
    \[\dim T(G)=3p^{k+n}+3p^{k+2n}+p^{2k}-6p^k+p^{4n}+3p^{3n}-5p^{2n}-6p^n+7.\]
\end{Prop}

\begin{proof}
    By Theorems \ref{thm:tanaka} \ref{thm:acclassify}, $T(G)$ is triply regular. Therefore, $\dim T(G)=|\{(i,j,h)\colon p_{ij}^h\neq 0\}|$. We shall count the number of triples such that $p_{ij}^h\neq 0$. By Proposition \ref{lem:pcamclass} the class of $x\in G$ is one of the following:
    \[x^G=xG'\ \text{ if } x\in G\setminus G',\ \ \ \ \ x^G=xZ(G)\ \text{ if }x\in G'\setminus Z(G), \ \ \ \ \ x^G=\{x\}\ \text{ if }x\in Z(G).\]
    Since $|Z(G)|=p^k$, $[G\colon G']=p^{2n}$, and $[G'\colon Z(G)]=p^n$, there are $p^k$ conjugacy classes with $x\in Z(G)$, $p^{n}-1$ conjugacy classes with $x\in G'\setminus Z(G)$, and $p^{2n}-1$ conjugacy classes with $x\in G\setminus G'$. This gives a total of $p^{2n}+p^n+p^k-2$ conjugacy classes. We shall denote these as $C_i$. We consider fixed $i,h$ and find the number of triples $(i,j,h)$ such that $p_{ij}^h\neq 0$. We consider cases:

    {\bf Case 1: } $i\neq h$. As $T(G)$ is almost commutative by Theorem \ref{thm:acclassify}, there exists a unique $j$ such that $p_{ij}^h\neq 0$. We have $p^{2n}+p^n+p^k-2$ choices for $C_i$. We then have $p^{2n}+p^n+p^k-3$ choices for $C_h$. Therefore, we have $(p^{2n}+p^n+p^k-2)(p^{2n}+p^n+p^k-3)$ triples $(i,j,h)$ with $p_{ij}^h\neq 0$ if $i\neq h$.

    {\bf Case 2: } $i=h$, with $C_i=\{x\}\subseteq Z(G)$. For $C_j\subseteq G'\setminus Z(G)$, we have $C_j=yZ(G)$ for some $y\in G'\setminus Z(G)$. Then $\overline{C_i}\cdot\overline{C_j}=\overline{yZ(G)}$. Note $x\not\in yZ(G)$ as $x\in Z(G)$. So $p_{ij}^h=0$ in this case. If $C_j=yG'\subseteq G\setminus G'$, then $\overline{C_i}\cdot\overline{C_j}=\overline{yG'}$. Now $x\not\in yG'$, so $p_{ij}^h=0$ in this case. Then the only time when possibly $p_{ij}^h\neq 0$ is if $C_j=\{y\}\subseteq Z(G)$. Then $\overline{C_i}\cdot\overline{C_j}=xy$. This contains $C_h=\{x\}$ if and only if $xy=x$, if and only if $y=e$. Thus, given any $i=h$ with $C_i\subseteq Z(G)$ there is a single triple, namely $(i,0,i)$ such that $p_{ij}^h\neq 0$. We have $p^k$ choices for $C_i=C_h$. This gives a total of $p^k$ triples with $p_{ij}^h\neq 0$.

    {\bf Case 3: }$i=h$ with $C_i=xZ(G)$, $x\in G'\setminus Z(G)$. If $C_j\subseteq G\setminus G'$, then $C_j=yG'$ for $y\in G\setminus G'$. Then $\overline{C_i}\cdot\overline{C_j}=\overline{yG'}$. Now $xZ(G)\not\subseteq yG'$, so $p_{ij}^h=0$ in this case. If $C_j\subseteq G'\setminus Z(G)$, then $C_j=yZ(G)$ for some $y\in G'\setminus Z(G)$. Then $\overline{C_i}\cdot\overline{C_j}=|Z(G)|\overline{xyZ(G)}$. Then $p_{ij}^h\neq 0$ if and only if $\overline{xyZ(G)}=\overline{xZ(G)}$, if and only if $y\in Z(G)$, which is false. So $p_{ij}^h=0$ in this case. Then the only time when $p_{ij}^h\neq 0$ is if $C_j=\{y\}\subseteq Z(G)$. Then $\overline{C_i}\cdot\overline{C_j}=\overline{xZ(G)}=\overline{C_h}$. So for all $y\in Z(G)$, we have $C_j=\{y\}$ such that $p_{ij}^h\neq 0$. There are $p^k$ choices for $C_j$ in this case. We have $p^n-1$ choices for $C_i=C_h$. Thus, we have a total of $p^k(p^n-1)$ triples $(i,j,h)$ with $p_{ij}^h\neq 0$ in this case.

    {\bf Case 4: }$i=h$ with $C_i=xG'$, $x\in G\setminus G'$. A similar argument to case 3 gives $(p^{2n}-1)(p^n+p^k-1)$ triples $(i,j,h)$ such that $p_{ij}^h\neq 0$ in this case.

    Having considered every possible case we get:
    \[\dim T(G)=(p^{2n}-1)(p^n+p^k-1)+p^k(p^n-1)+p^k+(p^{2n}+p^n+p^k-2)(p^{2n}+p^n+p^k-3)\]
    \[=3p^{k+n}+3p^{k+2n}+p^{2k}-6p^k+p^{4n}+3p^{3n}-5p^{2n}-6p^n+7.\qedhere\]
\end{proof}

 We now find the non-primary primitive central idempotents. We first determine the structure of the adjacency matrices $A_k$.

\begin{Lem}\label{lem:cam2outer}
    Let $C_k\not\subseteq G'$. Then for any classes $C_i,C_j$ of $G$ the $C_i,C_j$ block of $A_k$ is all $0$'s if $C_j\not\subseteq C_iC_k$ and all $1$'s if $C_j\subseteq C_iC_k$.
\end{Lem}

\begin{proof}
    As $C_k\not\subseteq G'$ and $G$ is a Camina group, $C_k=gG'$ for some $g\in G\setminus G'$. Now suppose $x\in C_i$ and $y\in C_j$. We have $(A_k)_{xy}=1$ if and only if $yx^{-1}\in gG'$, if and only if $y\in gG'x=gxG'$, if and only if $C_j\subseteq gxG'$. Now either $C_i\subseteq G'$ or $C_i=hG'$ for some $h\in G\setminus G'$. If $C_i\subseteq G'$, then $x\in G'$ and $gxG'=gG'=C_iC_k$. If $C_i=hG'$ for some $h\in G\setminus G'$, then $C_i=xG'$ as well. Then $gxG'=C_kC_i=C_iC_k$. So either way $C_iC_k=gxG'$. So for $x\in C_i$ and $y\in C_j$, $(A_k)_{xy}=1$ if and only if $C_j\subseteq C_iC_k$.
 
    Now suppose that the $C_i,C_j$ block of $A_k$ has a nonzero entry. Say that $(A_k)_{ab}=1$ for $a\in C_i, b\in C_j$. Then $ba^{-1}\in C_k$. Hence, $b\in C_ka\subseteq C_iC_k$. This implies $C_j\subseteq C_iC_k$, and so for all $x\in C_i$ and $y\in C_j$ that $(A_k)_{xy}=1$. Thus, the $C_i,C_j$ block of $A_k$ is all $1$'s, so the $C_i,C_j$ block of $A_k$ is all $0$'s or all $1$'s.
\end{proof}

\begin{Lem}\label{lem:cam3centeral}
    Let $C_k\subseteq Z(G)$ be a class of $G$. If $G$ has nilpotency class $2$, let $C_i,C_j$ be any two classes of $G$. If $G$ has nilpotency class $3$, then suppose $C_i,C_j\subseteq G'$ are classes of $G$. Then
    \begin{enumerate}
        \item If $C_i,C_j\subseteq Z(G)$, then the $C_i,C_j$ block of $A_k$ is $0$ if $C_j\neq C_iC_k$ and $1$ if $C_j=C_iC_k$.
        \item If $C_i\neq C_j$ with either $C_i\not\subseteq Z(G)$ or $C_j\not\subseteq Z(G)$. Then the $C_i,C_j$ block of $A_k$ is all $0$.
        \item If $C_i\not\subseteq Z(G)$, then the $C_i,C_i$ block of $A_k$ has a single nonzero entry in each row.
    \end{enumerate}
\end{Lem}

\begin{proof}
    As $C_k\subseteq Z(G)$, $C_k=\{g\}$.
    
    1: If $C_i,C_j\subseteq Z(G)$, then $C_i=\{x\}$ and $C_j=\{y\}$ for some $x,y\in Z(G)$. Then the $C_i,C_j$ block of $A_k$ is just the $(x,y)$ entry. So $(A_k)_{xy}=1$ if and only if $yx^{-1}\in C_k=\{g\}$. That is, $(A_k)_{xy}=1$ if and only if $yx^{-1}=g$, Hence, the $C_i,C_j$ block of $A_k$ is $0$ if $C_j\neq C_iC_k$ and $1$ if $C_j=C_iC_k$.

    For 2.: assume $C_i\neq C_j$ with either $C_i\not\subseteq Z(G)$ or $C_j\not\subseteq Z(G)$. Without loss of generality suppose $C_i\not\subseteq Z(G)$. Then by Proposition \ref{lem:pcamclass}, $C_i=hZ(G)$ for some $h\in G$. Let $x\in C_i$ and $y\in C_j$. If $(A_k)_{xy}=1$, then $yx^{-1}\in C_k\subseteq Z(G)$. We note as $x\in C_i=hZ(G)$, that $x^{-1}\in h^{-1}Z(G)$. If $C_j\subseteq Z(G)$, then $yx^{-1}\in h^{-1}Z(G)$. However, $h^{-1}Z(G)\cap Z(G)=\emptyset$, which contradicts $yx^{-1}\in h^{-1}Z(G)\cap Z(G)$. If $C_j\subseteq G\setminus Z(G)$, then by Proposition \ref{lem:pcamclass}, $C_j=\ell Z(G)$ for $\ell\in G\setminus Z(G)$. Since $C_i\neq C_j$, $\ell h^{-1}Z(G)\neq Z(G)$. As $y\in \ell Z(G)$ and $x^{-1}\in h^{-1}Z(G)$, $yx^{-1}\in \ell h^{-1}Z(G)$. However, $\ell h^{-1}Z(G)\cap Z(G)=\emptyset$ which contradicts $yx^{-1}\in h^{-1}Z(G)\cap Z(G)$. In either case we got a contradiction. Therefore, $(A_k)_{xy}\neq 1$. Thus for all $x\in C_i$ and $y\in C_j$ we have $(A_k)_{xy}=0$.

    For 3.: assume $C_i=C_j$ with $C_i\not\subseteq Z(G)$. Fix $x\in C_i$. Since $C_i\not\subseteq  Z(G)$ by Proposition \ref{lem:pcamclass}, $C_i=xZ(G)$. Then $(A_k)_{xy}=1$ if and only if $yx^{-1}\in C_k=\{g\}$, if and only if $y=xg$. Now as $C_i=xZ(G)$, we have $xg\in C_i$. So the $(x,xg)$ entry of $A_k$ is nonzero. This however is the only nonzero entry of the $x$th row of $A_k$. As we chose $x\in C_i$ arbitrarily, the $C_i,C_i$ block of $A_k$ has a single nonzero entry in each row. 
\end{proof}

\begin{Lem}\label{lem:cam3outer}
    Let $G$ have nilpotency class $3$. Let $C_k$ be a class of $G$ such that $C_k\subseteq G'$. Then for classes $C_i,C_j$,
    \begin{enumerate}
        \item If $C_i\neq C_j$ with either $C_i\not\subseteq G'$ or $C_j\not\subseteq G'$, then the $C_i,C_j$ block of $A_k$ is all $0$.
        \item If $C_i\not\subseteq G'$, then the $C_i,C_i$ block of $A_k$ has $|C_k|$ nonzero entries in each row.
    \end{enumerate}
\end{Lem}

\begin{proof}
    First consider 1. Without loss of generality say that $C_i\not\subseteq G'$. Then by Proposition \ref{lem:pcamclass}, $C_i=hG'$ for some $h\in G\setminus G'$. Let $x\in C_i$ and $y\in C_j$. If $(A_k)_{xy}=1$, then $yx^{-1}\in C_k\subseteq G'$. Since $x\in C_i=hG'$, $x^{-1}\in h^{-1}G'$. If $C_j\subseteq G'$, then $yx^{-1}\in h^{-1}G'$. However, $h^{-1}G'\cap G'=\emptyset$. This contradicts $yx^{-1}\in h^{-1}G'\cap G'$. If $C_j\not\subseteq G'$, then by Proposition \ref{lem:pcamclass}, $C_j=\ell G'$ for some $\ell\in G\setminus G'$. Since $C_i\neq C_j$, $\ell h^{-1}G'\neq G'$. As $y\in \ell G'$ and $x^{-1}\in h^{-1}G'$, $yx^{-1}\in \ell h^{-1}G'$. However, $\ell h^{-1}G'\cap G'=\emptyset$. This contradicts $yx^{-1}\in h^{-1}G'\cap G'$. In either case we got a contradiction. Therefore, $(A_k)_{xy}\neq 1$. Thus for all $x\in C_i$ and $y\in C_j$ we have $(A_k)_{xy}=0$.

    Next consider 2. Fix $x\in C_i$. Since $C_i\not\subseteq G'$ by Proposition \ref{lem:pcamclass}, $C_i=xG'$. We have $(A_k)_{xy}=1$ if and only if $yx^{-1}\in C_k$, if and only if $y\in xC_k\subseteq xG'=C_i$. Then for each $y\in xC_k$ we get a nonzero entry in row $x$ of $A_k$. As we chose $x\in C_i$ arbitrarily, the $C_i,C_i$ block of $A_k$ has $|C_k|$ nonzero entries in each row. 
\end{proof}

\begin{Lem}\label{lem:cam3inner}
    Let $G$ have nilpotency class $3$. Let $C_k$ be a class of $G$ such that $C_k\subseteq G'\setminus Z(G)$. If $C_i,C_j\subseteq G'$, then the $C_i,C_j$ block of $A_k$ is all $0$'s if $C_j\not\subseteq C_iC_k$ and all $1$'s if $C_j\subseteq C_iC_k$.
\end{Lem}

\begin{proof}
    The proof of this result is similar to the proof of Lemma \ref{lem:cam2outer}
\end{proof}

Using these four lemmas, we have the entire block structure of the adjacency matrices $A_k$ of $T(G)$. We use this as well as the fact $Z(G)$ is elementary abelian to construct the idempotents over the next several results.

\begin{Lem}\label{lem:nil2stuctrure1}
    Suppose that $|Z(G)|=p^k$, $k>0$, and $Z(G)=\langle z_1, z_2,\cdots, z_k\rangle$. If $G$ has nilpotency class $2$ then let $C_i$ be any class of $G$ outside of $G'$ and if $G$ has nilpotency class $3$ then let $C_i\subseteq G'\setminus Z(G)$ be a class of $G$. Let $C_{m,j}$ be the central class $\{z_m^j\}$ and $A_{m,j}$ be the adjacency matrix corresponding to $C_{m,j}$. Then $(E_i^*A_{m,1}E_i^*)^j=E_i^*A_{m,j}E_i^*$ for all $j\in \{0,1,2,\cdots p-1\}$.
\end{Lem}

\begin{proof}
    From the proof of Lemma \ref{lem:cam3centeral} we have for all $x,y\in C_i$, $(E_i^*A_{m,j} E_i^*)_{xy}=1$ if and only if $y=xz_m^j$. We shall proceed by induction on $j$. For $j=0,1$, the result is clear. We now suppose that the result is true for some $n$, $1\leq n\leq p-2$. That is, $(E_i^*A_{m,1}E_i^*)^n=E_i^*A_{m,n}E_i^*$. Then 
    \[(E_i^*A_{m,1}E_i^*)^{n+1}=(E_i^*A_{m,1}E_i^*)^n(E_i^*A_{m,1}E_i^*)=(E_i^*A_{m,n}E_i^*)(E_i^*A_{m,1}E_i^*).\]
    We observe $[(E_i^*A_{m,n}E_i^*)(E_i^*A_{m,1}E_i^*)]_{xy}=\sum_{u\in G}(E_i^*A_{m,n}E_i^*)_{xu}(E_i^*A_{m,1}E_i^*)_{uy}$.
    
    From Lemma \ref{lem:cam3centeral}, $(E_i^*A_{m,n}E_i^*)_{xu}(E_i^*A_{m,1}E_i^*)_{uy}=1$ if and only if $u=xz_m^n$ and $y=uz_m=xz_m^{n+1}$. The first of these guarantees the sum is either $0$ or $1$ as there is at most one nonzero term in this sum, namely, when $u=xz_m^n$. Hence, the $(x,y)$ entry of $(E_i^*A_{m,n}E_i^*)(E_i^*A_{m,1}E_i^*)$ is $1$ if and only if $y=xz_m^{n+1}$. From the proof of Lemma \ref{lem:cam3centeral} this is the condition for the $(x,y)$ entry of $E_i^*A_{m,n+1}E_i^*$ to be $1$. Hence, we have $[(E_i^*A_{m,n}E_i^*)(E_i^*A_{m,1}E_i^*)]_{xy}=(E_i^*A_{m,n+1}E_i^*)_{xy}$ for all $x,y\in C_{i}$. Clearly $(E_i^*A_{m,1}E_i^*)^{n+1}=E_i^*A_{m,n}E_i^*E_i^*A_{m,1}E_i^*$ is $0$ outside of the $C_i,C_i$ block. Hence, we have $E_i^*A_{m,n+1}E_i^*=(E_i^*A_{m,1}E_i^*)^{n+1}$. Then by induction we have $(E_i^*A_{m,1}E_i^*)^j=E_i^*A_{m,j}E_i^*$ for all $j\in \{0,1,\cdots p-1\}$. 
\end{proof}

\begin{Lem}\label{lem:nil2structure2}
    Suppose that $|Z(G)|=p^k$, $k>0$, and $Z(G)=\langle z_1, z_2,\cdots, z_k\rangle$. If $G$ has nilpotency class $2$ then let $C_i$ be any class of $G$ outside of $G'$; if $G$ has nilpotency class $3$ then let $C_i\subseteq G'\setminus Z(G)$ be a class of $G$. Suppose that for $1\leq j\leq k$ we have $C_j=\{z_j\}$. 
    Let $C_\ell=\{z_1^{m_1}z_2^{m_2}\cdots z_k^{m_k}\}$ be any central class of $G$. Then $E_i^*A_\ell E_i^*=\prod_{j=1}^k (E_i^*A_{j}E_i^*)^{m_j}$.
\end{Lem}

\begin{proof}
    Let $L$ be the number of nonzero $m_j$ in $C_\ell=\{z_1^{m_1}z_2^{m_2}\cdots z_k^{m_k}\}$ where $0\leq m_i< p$. We proceed by induction on $L$. For $L=0$, the result is true as both sides are the identity in the $C_i,C_i$ block and $0$ outside of this block. If $L=1$, then without the loss of generality we may assume that $m_1$ is the only nonzero $m_j$. Then 
    \[\prod_{j=1}^k (E_i^*A_{j}E_i^*)^{m_j}=(E_i^*A_{j}E_i^*)^{m_1}.\]
    By Lemma \ref{lem:nil2stuctrure1}, this equals $E_i^*A_\ell E_i^*$, so the result is true for $L=1$. Now assume the result is true for some $n$ with $1\leq n\leq k-1$. That is, for all $C_\ell=\{z_1^{m_1}z_2^{m_2}\cdots z_k^{m_k}\}$ with $n$ of the $m_j\neq 0$ we have
    \[E_i^*A_\ell E_i^*=\prod_{j=1}^k (E_i^*A_{j}E_i^*)^{m_j}.\]
    Now we consider $C_\alpha=\{z_1^{a_1}z_2^{a_2}\cdots z_k^{a_k}\}$ in which $n+1$ of the $a_j$ are not zero. Without loss of generality, we may assume the first $n+1$ of the $a_j$ are not zero. Then we have
    \[\prod_{j=1}^k (E_i^*A_{j}E_i^*)^{a_j}=\prod_{j=1}^{n+1} (E_i^*A_jE_i^*)^{a_j}=\left(\prod_{j=1}^{n} (E_i^*A_jE_i^*)^{a_j}\right)(E_i^*A_{n+1}E_i^*)^{a_{n+1}}.\]
    Letting $C_\beta=\{z_1^{a_1}z_2^{a_2}\cdots z_n^{a_n}\}$ we have
    \[\left(\prod_{j=1}^{n} (E_i^*A_jE_i^*)^{a_j}\right)(E_i^*A_{n+1}E_i^*)^{a_{n+1}}=(E_i^*A_\beta E_i^*)(E_i^*A_{n+1}E_i^*)^{a_{n+1}}.\]
    Letting $C_\gamma=\{z_{n+1}^{a_{n+1}}\}$ by Lemma \ref{lem:nil2stuctrure1}, $(E_i^*A_{n+1}E_i^*)^{a_{n+1}}=E_i^*A_\gamma E_i^*$. We then have
    \[\prod_{j=1}^k (E_i^*A_{j}E_i^*)^{a_j}=(E_i^*A_\beta E_i^*)(E_i^*A_\gamma E_i^*).\]

    Now we observe $[(E_i^*A_{\beta}E_i^*)(E_i^*A_{\gamma}E_i^*)]_{xy}=\sum_{u\in G}(E_i^*A_{\beta}E_i^*)_{xu}(E_i^*A_{\gamma}E_i^*)_{uy}$. From Lemma \ref{lem:cam3centeral}, $(E_i^*A_{\beta}E_i^*)_{xu}(E_i^*A_{\gamma}E_i^*)_{uy}=1$ if and only if $u=xz_1^{a_1}z_2^{a_2}\cdots z_n^{a_n}$ and $y=uz_{n+1}^{a_{n+1}}=xz_1^{a_1}z_2^{a_2}\cdots z_{n+1}^{a_{n+1}}$. The first of these guarantees the sum is either $0$ or $1$ as there is at most one nonzero term in this sum, namely, when $u=xz_1^{a_1}z_2^{a_2}\cdots z_n^{a_n}$. Hence, the $(x,y)$ entry of $(E_i^*A_{\beta}E_i^*)(E_i^*A_{\gamma}E_i^*)$ is $1$ if and only if $y=xz_1^{a_1}z_2^{a_2}\cdots z_{n+1}^{a_{n+1}}$. From the proof of Lemma \ref{lem:cam3centeral} this is the condition for the $(x,y)$ entry of $E_i^*A_{\alpha}E_i^*$ to be $1$. Hence, we have $[(E_i^*A_{\beta}E_i^*)(E_i^*A_{\gamma}E_i^*)]_{xy}=(E_i^*A_{\alpha}E_i^*)_{xy}$ for all $x,y\in C_{i}$. Hence, we have $\prod_{j=1}^k (E_i^*A_{j}E_i^*)^{a_j}=E_i^*A_\alpha E_i^*$. Then by induction $E_i^*A_\ell E_i^*=\prod_{j=1}^k (E_i^*A_{j}E_i^*)^{m_j}$.
\end{proof}

Before constructing the primitive central idempotents, we prove one more result about the nature of the blocks of the adjacency matrices outside of $G'$ when considering an adjacency matrix coming from a class in the center of the group.

\begin{Prop}\label{prop:kroneckerdec2}
    Suppose that $|Z(G)|=p^k$, $k>0$, and $Z(G)=\langle z_1, z_2,\cdots, z_k\rangle$. If $G$ has nilpotency class $2$ then let $C_j$ be any class of $G$ outside of $G'$ and if $G$ has nilpotency class $3$ then let $C_j\subseteq G'\setminus Z(G)$. Let $C_i=\{z_i\}$. Then the $C_j,C_j$ block of $A_i$ can be written as $\bigotimes_{m=1}^k M_m$ (the Kronecker product) where each $M_m$ is a $p\times p$ matrix with $M_m=I$ for $m\neq k-i+1$ and
    \begin{equation}\label{eq:tensor}
    M_{k-i+1}=\begin{pmatrix}
        0 & 1 & 0 & \cdots & 0 \\
        0 & 0 & 1 & \cdots & 0 \\
        \vdots & \vdots & & \ddots & \vdots \\
        0 & 0 & 0 & \cdots & 1 \\
        1 & 0 & 0 & \cdots & 0
    \end{pmatrix}\end{equation}
    relative to a specific ordering of the elements $Z(G)$.
\end{Prop}

\begin{proof}
    Let $A_i'$ denote the $C_j,C_j$ block of $A_i$. By Proposition \ref{lem:pcamclass}, $C_j=xZ(G)$. We order the elements of $Z(G)$ as follows. First let $\omega$ be the sequence $1,z_1,z_1^2,\cdots, z_1^{p-1}$. We define $\omega(m_2,m_3,\cdots, m_k)=z_k^{m_k}z_{k-1}^{m_{k-1}}\cdots z_2^{m_2}\cdot\omega$. Now we order the elements of $Z(G)$ as 
    \[\omega, \omega(1,0,\cdots, 0),\omega(2,0,\cdots, 0), \cdots, \omega(p-1,0,\cdots, 0), \omega(0,1,\cdots, 0), \omega(1,1,\cdots, 0),\cdots, \omega(p-1,1,\cdots, 0),\]
    \[\omega(0,2,\cdots 0), \omega(1,2,\cdots 0),\cdots,\omega(p-1,p-1,0\cdots, 0),\omega(0,0,1,\cdots 0),\omega(1,0,1,\cdots, 0),\cdots,  \omega(p-1,\cdots, p-1).\]

    We proceed by induction on $k$. For $k=1$, we have $Z(G)=\{1,z_1,z_1^2,\cdots, z_1^{p-1}\}$. Consider the row indexed by $xz_1^{m_1}$ in $A_1'$. From the proof of Lemma \ref{lem:cam3centeral} the only nonzero entry in this row is then $(xz_1^{m_1},xz_1^{m_1+1})$ entry. Notice that each nonzero entry of $A_1'$ is shifted one to the right of the main diagonal working modulo $p$ relative to the ordering of $Z(G)$. Therefore, 
    \[A_1'=\begin{pmatrix}
        0 & 1 & 0 & \cdots & 0 \\
        0 & 0 & 1 & \cdots & 0 \\
        \vdots & \vdots & & \ddots & \vdots \\
        0 & 0 & 0 & \cdots & 1 \\
        1 & 0 & 0 & \cdots & 0
    \end{pmatrix}.\]
   This is equal to $\bigotimes_{m=1}^k M_m=\bigotimes_{m=1}^1 M_m=M_1$ where $M_1$ is the matrix in (\ref{eq:tensor}). This proves the base case.

    Now suppose that the result is true for $k=n$. We consider the $n+1$ case. Let us first consider $A_i'$ for $1\leq i\leq n$. Consider the $x\omega(m_2,m_3,\cdots, m_{n+1})$ row block of $A_i'$. Row $xz_{n+1}^{m_{n+1}}z_n^{m_n}\cdots z_2^{m_2}z_1^{m_1}$ in this row block has a single nonzero entry in entry $(xz_{n+1}^{m_{n+1}}z_n^{m_n}\cdots z_2^{m_2}z_1^{m_1}, xz_{n+1}^{m_{n+1}}z_n^{m_n}\cdots z_i^{m_i+1} z_{i-1}^{m_{i-1}}\cdots z_2^{m_2}z_1^{m_1})$ by Lemma \ref{lem:cam3centeral}. The only nonzero column block in this row block is the block $x\omega(m_2,m_3,\cdots, m_i+1,m_{i+1},\cdots, m_n,m_{n+1})$. For each choice of $m_{n+1}$, the nonzero entry is \\$(xz_{n+1}^{m_{n+1}}z_n^{m_n}\cdots z_2^{m_2}z_1^{m_1}, xz_{n+1}^{m_{n+1}}z_n^{m_n}\cdots z_i^{m_i+1} z_{i-1}^{m_{i-1}}\cdots z_2^{m_2}z_1^{m_1}).$ 
    Thus each $x\omega(m_2,m_3,\cdots, m_{n+1})$,\\ $x\omega(m_2,\cdots, m_{i}+1,m_{i+1},\cdots, m_{n+1})$ block of $A_i'$ is the same. All other blocks are $0$ in the $x\omega(m_2,m_3,\cdots, m_{n+1})$ row block. As each of these blocks have the same power of $z_{n+1}$, they are all diagonal blocks because of the ordering of $Z(G)$. This means $A_i'=I\otimes M$ where $M$ is some $p^n\times p^n$ matrix. In $M$ we are holding the power on $z_{n+1}$ fixed and only considering $z_1,z_2,\cdots, z_n$. Relabeling in $M$ by dropping the $z_{n+1}^{m_{n+1}}$ term, does not change $M$. Notice the $(xz_n^{m_n}z_{n-1}^{m_{n-1}}\cdots z_1^{m_1},xz_{n}^{m_n}\cdots z_i^{m_i+1} z_{i-1}^{m_{i-1}}\cdots z_1^{m_1})$ entry is the only nonzero entry in row $xz_n^{m_n}x_{n-1}^{m_{n-1}}\cdots z_1^{m_1}$. As $G/\langle z_{n+1}\rangle$ is a Camina group by Proposition \ref{prop:cammod}, we have by Lemma \ref{lem:cam3centeral}, entry $(xz_n^{m_n}z_{n-1}^{m_{n-1}}\cdots z_1^{m_1},xz_{n}^{m_n}\cdots z_i^{m_i+1} z_{i-1}^{m_{i-1}}\cdots z_1^{m_1})$ is the only nonzero entry in row $xz_n^{m_n}x_{n-1}^{m_{n-1}}\cdots z_1^{m_1}$ of $\overline{A}_i$, where $\overline{A}_i$ is the $xZ(G/\langle z_{n+1}\rangle),xZ(G/\langle z_{n+1}\rangle)$ block of $A_i$ in $G/\langle z_{n+1}\rangle$. Therefore, $M$ and $\overline{A}_i$ are the same in every row, so $M=\overline{A}_i$. By the inductive hypothesis $\overline{A_i}=\bigotimes_{m=1}^n N_m$ where $N_m=I$ for $m\neq n-i+1$ and $N_{n-i+1}$ is the matrix in (\ref{eq:tensor}). We then have $A_i'=I\otimes \bigotimes_{m=1}^n N_m$. Shifting all the indices $N_m$ up by one and setting $I=N_1$ we get $A_i'=\bigotimes_{m=1}^{n+1} N_m$ where $N_M=I$ for $j\neq n+1-i+1$ and $N_{(n+1)-i+1}$ is the matrix in (\ref{eq:tensor}). Hence, $A_i'$ has the desired form.

    Now we consider $A_{n+1}'$. Consider the $x\omega(m_2,m_3,\cdots, m_{n+1})$ row block. Row $xz_{n+1}^{m_{n+1}}\cdots z_1^{m_1}$ has nonzero entry $(xz_{n+1}^{m_{n+1}}\cdots z_1^{m_1},xz_{n+1}^{m_{n+1}+1}\cdots z_1^{m_1})$ by Lemma \ref{lem:cam3centeral}. Then the only nonzero column block in this row block is $x\omega(m_2,m_3,\cdots, m_{n+1}+1)$. For each choice of $m_{n+1}$ the 
    \[x\omega(m_2,m_3,\cdots, m_{n+1}),x\omega(m_2,m_3,\cdots, m_{n+1}+1)\] 
    block of $A_{n+1}'$ is the same. All other blocks are $0$. Notice fixing $m_{n+1}$, we can relabel the entries of the $x\omega(m_2,m_3,\cdots, m_{n+1}),x\omega(m_2,m_3,\cdots, m_{n+1}+1)$ block without $z_{n+1}$ since it does not affect entries in this block as we are fixing it. Doing this relabeling the only nonzero entries in this block are of the form $(xz_n^{m_n}\cdots z_1^{m_1},xz_n^{m_n}\cdots z_1^{m_1})$. Thus the matrix in the $x\omega(m_2,m_3,\cdots, m_{n+1}),x\omega(m_2,m_3,\cdots, m_{n+1}+1)$ block is $I_{p^n}=\bigotimes_{m=1}^n I_p$ for each block. Each of these nonzero blocks is shifted one to the right of the main diagonal as each has one higher power on $z_{n+1}$ for the column than the row. We then have $A_{n+1}'=M\otimes \bigotimes_{m=1}^n I_p$ where $M$ is the matrix in (\ref{eq:tensor}). So $A_{n+1}'$ has the desired form. This completes the inductive step. Then by induction, $\bigotimes_{m=1}^k M_m$ where $M_m=I$ for $m\neq k-i+1$ and $M_{k-i+1}$ is the matrix in (\ref{eq:tensor}) for all $i$.
\end{proof}

We are now ready to construct the primitive central idemopotents for the nilpotency class $2$ case. In doing so we first prove that the constructed elements are idemopotents, and then prove they are all orthogonal. We write $A\sim B$ if $A$ and $B$ are similar matrices and use $\trace(A)$ to denote the trace of matrix $A$. We let $\langle A,B\rangle$ denote the inner product given by $\langle A,B\rangle=\trace(A\overline{B}^T)$.

\begin{Thm}\label{thm:2idm}
    Suppose $|Z(G)|=p^k$ with $Z(G)=\langle z_1,z_2,\cdots, z_k\rangle$. Let $C_j=\{z_j\}$ be the class containing $z_j$ for all $1\leq j\leq k$. If $G$ has nilpotency class $2$ then let $C_i$ be any class of $G$ outside of $G'$ and if $G$ has nilpotency class $3$ then let $C_i\subseteq G'\setminus Z(G)$. Let $\zeta$ be any primitive $p$th root of unity. Then
    \[W_i(\alpha)=\frac{1}{p^k}\sum_{a_1,a_2,\cdots, a_k\in \{0,1,\cdots, p-1\}} \prod_{j=1}^k \zeta^{\alpha_j a_j}(E_i^*A_jE_i^*)^{a_j}\]
    is an idempotent of $T(G)$ for any $\alpha=(\alpha_1,\alpha_2,\cdots, \alpha_k)$ with each $\alpha_j\in \{0,1,\cdots, p-1\}$.
\end{Thm}

\begin{proof}
    Fix $\alpha=(\alpha_1,\alpha_2,\cdots, \alpha_k)$ with each $\alpha_j\in \{0,1,2\cdots, p-1\}$. We note $W_i(\alpha)$ is only nonzero in the $C_i,C_i$ block. If we let $A_j'$ denote the $C_i,C_i$ block of $A_j$ the $C_i,C_i$ block of $W_i(\alpha)$ is just 
    \[W_i'(\alpha)=\frac{1}{p^k}\sum_{a_1,a_2,\cdots, a_k\in \{0,1,\cdots, p-1\}} \prod_{j=1}^k \zeta^{\alpha_j a_j}(A_j')^{a_j}.\]
    Hence, $W_i(\alpha)$ is an idempotent if and only if $W_i'(\alpha)$
    is an idempotent. By Proposition \ref{prop:kroneckerdec2}, $A_j'=\bigotimes_{m=1}^k M_{m,j}$ where $M_{m,j}=I$ for $m\neq k-j+1$ and $M_{k-j+1,j}$ is the matrix in (\ref{eq:tensor}), which we denote as $P$.
    Then
    \[W_i'(\alpha)=\frac{1}{p^k}\sum_{a_1,a_2,\cdots, a_k\in \{0,1,\cdots, p-1\}}\prod_{j=1}^k \zeta^{\alpha_j a_j} \bigotimes_{m=1}^k M_{m,j}^{a_j}.\]
    As $M_{m,j}=I$ for all $m$ expect $k-j+1$ we have
    \[\frac{1}{p^k}\sum_{a_1,a_2,\cdots, a_k\in \{0,1,\cdots, p-1\}}\prod_{j=1}^k \zeta^{\alpha_j a_j} \bigotimes_{m=1}^k M_{m,j}^{a_j}=\frac{1}{p^k}\sum_{a_1,a_2,\cdots, a_k\in \{0,1,\cdots, p-1\}} \zeta^{\alpha_j a_j} \bigotimes_{j=1}^k P^{a_j}.\]
    This last equality holds since for each term in the product we can move the coefficient $\zeta^{\alpha_j a_j}$ to be on the one non-identity term in the Kronecker product for each $j$. Then in the product we have in each position of the Kronecker product $\zeta^{\alpha_j a_j}P_j$ multiplied by the identity repeatedly. 
    
    Note that the matrix $P$ is diagaonlizable and is similar to 
    \begin{equation}\label{eq:diag}
        D=\begin{pmatrix}
        1 & & & \\  & \zeta & & \\ & & & \ddots & \\ & & & & \zeta^{p-1} 
    \end{pmatrix}.\end{equation}
    Say that $D=QPQ^{-1}$. Then conjugating by $\bigotimes_{m=1}^k Q$ we have
    \[\frac{1}{p^k}\sum_{a_1,a_2,\cdots, a_k\in \{0,1,\cdots, p-1\}} \zeta^{\alpha_j a_j} \bigotimes_{j=1}^k P^{a_j}\sim
    \bigotimes_{m=1}^k Q \left(\frac{1}{p^k}\sum_{a_1,a_2,\cdots, a_k\in \{0,1,\cdots, p-1\}} \zeta^{\alpha_j a_j}\bigotimes_{j=1}^k P^{a_j}\right) \bigotimes_{m=1}^k Q^{-1}\]
    \[=\frac{1}{p^k}\sum_{a_1,a_2,\cdots, a_k\in \{0,1,\cdots, p-1\}} \zeta^{\alpha_j a_j}\bigotimes_{j=1}^k QP^{a_j}Q^{-1}=\frac{1}{p^k}\sum_{a_1,a_2,\cdots, a_k\in \{0,1,\cdots, p-1\}} \zeta^{\alpha_j a_j}\bigotimes_{j=1}^k D^{a_j}.\]

    The only nonzero entries in this matrix are the diagonal entries. Note that in $\bigotimes_{m=1}^k D$ that each entry is a product of powers of $\zeta$. With this observation we let $(t_1,t_2,\cdots, t_k)$ denote the entry of $\bigotimes_{m=1}^k D$ corresponding to $\zeta^{t_1}\zeta^{t_2}\cdots \zeta^{t_k}$. With this notation the $(t_1,t_2,\cdots, t_k)$ entry of the above matrix is 
    \[\frac{1}{p^k}\sum_{a_1,a_2\cdots a_k} \zeta^{a_1\alpha_1}\zeta^{a_1t_1}\zeta^{a_2\alpha_2}\zeta^{a_2t_2}\cdots \zeta^{a_k\alpha_k}\zeta^{a_kt_k}=\frac{1}{p^k}\sum_{a_1,a_2,\cdots, a_k} \zeta^{a_1(\alpha_1+t_1)}\zeta^{a_2(\alpha_2+t_2)}\cdots \zeta^{a_k(\alpha_k+t_k)}.\]
    If $\alpha_j+t_j\not\equiv 0 \mod p$ for any $j$, then as we range over all $a_j\in \{0,1,\cdots, p-1\}$ we get each root of unity appearing in the sum. Therefore in this case
    \[\frac{1}{p^k}\sum_{a_1,a_2,\cdots, a_k} \zeta^{a_1(\alpha_1+t_1)}\zeta^{a_2(\alpha_2+t_2)}\cdots \zeta^{a_k(\alpha_k+t_k)}\]
    \[=\frac{1}{p^k}\sum_{a_j} \zeta^{a_j(\alpha_j+t_j)}\sum_{a_1,\cdots, a_{j-1},a_{j+1},\cdots a_k} \zeta^{a_1(\alpha_1+t_1)}\cdots \zeta^{a_{j-1}(\alpha_{j-1}+t_{j-1})}\zeta^{a_{j+1}(\alpha_{j+1}+t_{j+1})}\cdots \zeta^{a_{k}(\alpha_{k}+t_{k})}=0.\]
    To get a nonzero term we must have $t_j\equiv -\alpha_j \mod p$ for all $j$. As the $\alpha_j$ are fixed this means that there is a single nonzero diagonal entry in
    \[\frac{1}{p^k}\sum_{a_1,a_2,\cdots, a_k\in \{0,1,\cdots, p-1\}} \bigotimes_{j=1}^k \left(\zeta^{\alpha_j}D\right)^{a_j}.\]
    This is the $(-\alpha_1,\-\alpha_2,\cdots, -\alpha_k)$ entry (where we are taking each of these modulo $p$). This entry is $\frac{1}{p^k}\sum_{a_1,a_2,\cdots, a_k} 1=1$.
    
    We then have
    \[\left(\frac{1}{p^k}\sum_{a_1,a_2,\cdots, a_k\in \{0,1,\cdots, p-1\}} \bigotimes_{j=1}^k \left(\zeta^{\alpha_j}D\right)^{a_j}\right)^2=\frac{1}{p^k}\sum_{a_1,a_2,\cdots, a_k\in \{0,1,\cdots, p-1\}} \bigotimes_{j=1}^k \left(\zeta^{\alpha_j}D\right)^{a_j},\]
    so this matrix is an idempotent. Then as $W_i'(\alpha)$ is similar to this matrix, it must be an idempotent as well. We then have that $W_i(\alpha)$ is an idempotent since its only nonzero block is the $C_i,C_i$ block, which we found to be an idempotent.
\end{proof}

\begin{Prop}\label{prop:orth2}
    Suppose $|Z(G)|=p^k$ with $Z(G)=\langle z_1,z_2,\cdots, z_k\rangle$. If $G$ has nilpotency class $2$ then let $C_i$ be any conjugacy class of $G$ outside of $G'$ and if $G$ has nilpotency class $3$ then let $C_i\subseteq G'\setminus Z(G)$. Let $\zeta$ be any primitive $p$th root of unity. Then using the notation from Theorem \ref{thm:2idm} we have that if $\alpha\neq \beta$, then $W_i(\alpha)$ is orthogonal to $W_i(\beta)$.
\end{Prop}

\begin{proof}
    Let $\alpha=(\alpha_1,\alpha_2,\cdots, \alpha_k)$ and $\beta=(\beta_1,\beta_2,\cdots, \beta_k)$ with $\alpha_j,\beta_j\in \{0,1,\cdots, p-1\}$ for all $j$. Then
    \[W_i(\alpha)=\frac{1}{p^k}\sum_{a_1,a_2,\cdots, a_k} \prod_{j=1}^k \zeta^{\alpha_j a_j}(E_i^*A_jE_i^*)^{a_j} \text{ and } W_i(\beta)=\frac{1}{p^k}\sum_{a_1,a_2,\cdots, a_k} \prod_{j=1}^k \zeta^{\beta_j a_j}(E_i^*A_jE_i^*)^{a_j}.\]
    Since both $W_i(\alpha)$ and $W_i(\beta)$ are $0$ outside of the $C_i,C_i$ block, the inner product of $W_i(\alpha)$ and $W_i(\beta)$ is the same as if we only considered the $C_i,C_i$ block of each. For this reason we set
    \[W_i'(\alpha)=\frac{1}{p^k}\sum_{a_1,a_2,\cdots, a_k\in \{0,1,\cdots, p-1\}} \prod_{j=1}^k \zeta^{\alpha_j a_j}(A_j')^{a_j} \text{ and }W_i'(\beta)=\frac{1}{p^k}\sum_{a_1,a_2,\cdots, a_k\in \{0,1,\cdots, p-1\}} \prod_{j=1}^k \zeta^{\beta_j a_j}(A_j')^{a_j}\]
    where $A_j'$ is the $C_i,C_i$ block of the matrix $A_j$.
    Let $P$ be the $p\times p$ matrix in (\ref{eq:tensor}).

    Using the same argument as in the proof of Theorem \ref{thm:2idm} we have  
    \[W_i'(\alpha)=\frac{1}{p^k}\sum_{a_1,a_2,\cdots, a_k\in \{0,1,\cdots, p-1\}} \zeta^{\alpha_j a_j}\bigotimes_{j=1}^k P^{a_j}\text{ and }W_i'(\beta)=\frac{1}{p^k}\sum_{a_1,a_2,\cdots, a_k\in \{0,1,\cdots, p-1\}} \zeta^{\beta_j a_j}\bigotimes_{j=1}^k P^{a_j}.\]
    Note that the matrix $P$ is similar to the matrix $D$ in (\ref{eq:diag}).

    Say $D=QPQ^{-1}$. Let $Q'=\bigotimes_{m=1}^k Q$. Then we note that $Q'W_i'(\alpha)\overline{W_i'(\beta)^T}(Q')^{-1}\sim W_i'(\alpha)\overline{W_i'(\beta)^T}$, so we have
    \[\langle W_i'(\alpha),W_i'(\beta)\rangle=\trace(W_i'(\alpha)\overline{W_i'(\beta)^T})=\trace(Q'W_i'(\alpha)\overline{W_i'(\beta)^T}(Q')^{-1})=\trace(Q'W_i'(\alpha)(Q')^{-1}Q'\overline{W_i'(\beta)^T}(Q')^{-1}).\]
    By the same argument as in the proof of Theorem \ref{thm:2idm}
    \[Q'W_i'(\alpha)(Q')^{-1}=\frac{1}{p^k}\sum_{a_1,a_2,\cdots, a_k\in \{0,1,\cdots, p-1\}} \zeta^{\alpha_j a_j}\bigotimes_{j=1}^k D^{a_j}\]
    and
    \[Q'W_i'(\beta)(Q')^{-1}=\frac{1}{p^k}\sum_{a_1,a_2,\cdots, a_k\in \{0,1,\cdots, p-1\}} \zeta^{\beta_j a_j}\bigotimes_{j=1}^k D^{a_j}.\]
    Hence, $\langle W_i'(\alpha),W_i'(\beta)\rangle$ equals
    \[\trace\left(\left(\frac{1}{p^k}\sum_{a_1,a_2,\cdots, a_k\in \{0,1,\cdots, p-1\}} \zeta^{\alpha_j a_j}\bigotimes_{j=1}^k D^{a_j}\right)\left(\overline{\frac{1}{p^k}\left(\sum_{a_1,a_2,\cdots, a_k\in \{0,1,\cdots, p-1\}} \zeta^{\beta_j a_j}\bigotimes_{j=1}^k D^{a_j}\right)^T}\right)\right).\]
    Using the notation from Theorem \ref{thm:2idm} the matrices
    \[\frac{1}{p^k}\sum_{a_1,a_2,\cdots, a_k\in \{0,1,\cdots, p-1\}} \zeta^{\alpha_j a_j}\bigotimes_{j=1}^k D^{a_j}\text{ and }\frac{1}{p^k}\sum_{a_1,a_2,\cdots, a_k\in \{0,1,\cdots, p-1\}} \zeta^{\beta_j a_j}\bigotimes_{j=1}^k D^{a_j}\]
    are diagonal matrices with a single nonzero entry of $1$ in diagonal positions $(-\alpha_1,-\alpha_2,\dots, -\alpha_k)$ and $(-\beta_1,-\beta_2,\cdots, -\beta_k)$ respectively (where each of these terms is found modulo $p$). This implies
    \[\overline{\frac{1}{p^k}\left(\sum_{a_1,a_2,\cdots, a_k\in \{0,1,\cdots, p-1\}} \zeta^{\beta_j a_j}\bigotimes_{j=1}^k D^{a_j}\right)^T}=\overline{\frac{1}{p^k}\sum_{a_1,a_2,\cdots, a_k\in \{0,1,\cdots, p-1\}} \zeta^{\beta_j a_j}\bigotimes_{j=1}^k D^{a_j}}\]
    \[=\frac{1}{p^k}\sum_{a_1,a_2,\cdots, a_k\in \{0,1,\cdots, p-1\}} \zeta^{\beta_j a_j}\bigotimes_{j=1}^k D^{a_j}.\]
    By assumption $\alpha\neq \beta$, so $(-\alpha_1,-\alpha_2,\cdots, -\alpha_k)\neq (-\beta_1,-\beta_2,\cdots, -\beta_k)$. This then implies 
    \[\left(\frac{1}{p^k}\sum_{a_1,a_2,\cdots, a_k\in \{0,1,\cdots, p-1\}} \zeta^{\alpha_j a_j}\bigotimes_{j=1}^k D^{a_j}\right)\left(\overline{\frac{1}{p^k}\left(\sum_{a_1,a_2,\cdots, a_k\in \{0,1,\cdots, p-1\}} \zeta^{\beta_j a_j}\bigotimes_{j=1}^k D^{a_j}\right)^T}\right)\]
    \[=\left(\frac{1}{p^k}\sum_{a_1,a_2,\cdots, a_k\in \{0,1,\cdots, p-1\}} \zeta^{\alpha_j a_j}\bigotimes_{j=1}^k D^{a_j}\right)\left(\frac{1}{p^k}\sum_{a_1,a_2,\cdots, a_k\in \{0,1,\cdots, p-1\}} \zeta^{\beta_j a_j}\bigotimes_{j=1}^k D^{a_j}\right)=0\]
    as these diagonal matrices have different nonzero diagonal entries. Therefore, $\langle W_i'(\alpha),W_i'(\beta)\rangle=0$. Recall this equals $\langle W_i(\alpha),W_i(\beta)\rangle$. Hence, $W_i(\alpha)$ and $W_i(\beta)$ are orthogonal.
\end{proof}

These turn out to be the only central idempotents in the nilpotency class $2$ case. As such, we can now prove the Wedderburn decomposition of the Terwilliger algebra in this case.

\begin{Thm}\label{thm:cam2decomp}
    Let $G$ have nilpotency class $2$ with $|G|=p^n$. Suppose $|Z(G)|=p^k$ Then the Wedderburn decomposition of $T(G)$ is
    \[T(G)=V\oplus \bigoplus_{i=1}^{p^{n-k}-1}\bigoplus_{\alpha}\mathcal{W}_i(\alpha),\]
    where $V$ is the primary component and each $\mathcal{W}_i(\alpha)$ is a one-dimensional ideal generated by a matrix $W_i(\alpha)$ (using the notation from Theorem \ref{thm:2idm}), where $\alpha=(\alpha_1,\alpha_2,\cdots, \alpha_k)$ with each $\alpha_j\in \{0,1,2,\cdots, p-1\}$ and not all $\alpha_j=0$.
\end{Thm}

\begin{proof}
    By Proposition \ref{cor:p2dim}, $\dim T(G)=p^{2n-2k}+3p^n-3p^{n-k}+p^{2k}-3p^k+2$. We know $V$ is an irreducible ideal with dimension equal to the number of conjugacy classes of $G$ squared. From the proof of Proposition \ref{cor:p2dim}, $\dim(V)=(p^{n-k}-1+p^k)^2$. We can then write $T(G)=V\oplus R$ where $\dim(R)=p^n-p^k-p^{n-k}+1=(p^k-1)(p^{n-k}-1)$. 

    From Proposition \ref{prop:orth2} we have, for each choice of $C_i\not\subseteq G'$, idempotents $W_i(\alpha)$ all of which are orthogonal. Clearly for $i\neq j$ we have $W_i(\alpha)$ and $W_j(\beta)$ are orthogonal as they are nonzero in different blocks. For any fixed $C_i\not\subseteq G'$, we consider $W_i(\alpha)$ when $\alpha=(0,0,\cdots, 0)$. We have
    \[W_i(\alpha)=\frac{1}{p^k}\sum_{a_1,a_2,\cdots, a_k\in \{0,1,\cdots, p-1\}} \prod_{j=1}^k (E_i^*A_jE_i^*)^{a_j}.\]
    By Lemma \ref{lem:nil2structure2}, $\prod_{j=1}^k (E_i^*A_jE_i^*)^{a_j}=E_i^*A_{\ell}E_i^*$ where $C_\ell=\{z_1^{a_1}z_2^{a_2}\cdots z_k^{a_k}\}$. This means $W_i(\alpha)$ is a scaled sum of the $C_i,C_i$ blocks of all the adjacency matrices corresponding to an element in $G'$. 

    Let $C_t$ be any class of $G$ outside of $G'$. Say $C_i=xZ(G)$ and $C_t=yZ(G)$. Then $C_iC_t=xyZ(G)$. Suppose $x\in xyZ(G)$. Then $x=xyz$ for some $z\in Z(G)$ and then $y^{-1}=z$. This implies $y^{-1}\in Z(G)$, which is false. Therefore, $C_i\not\subseteq C_iC_t$. By Lemma \ref{lem:cam2outer} the $C_i,C_i$ block of $A_t$ is all $0$. Hence, every adjacency matrix corresponding to a class outside $G'$ is $0$ in the $C_i,C_i$ block. As the sum of all the adjacency matrices is the all $1$'s matrix, the sum of all the adjacency matrices corresponding to classes outside $G'$ must be $0$ in the $C_i,C_i$ block. Hence, the sum of all the adjacency matrices corresponding to $C_t\subseteq G'$ must be the all $1$'s matrix in the $C_i,C_i$ block. Hence, $W_i(\alpha)$ has all $1$'s in the $C_i,C_i$ block and is $0$ outside of this block. Therefore, it is equal to the basis element of $V$ that is the all $1$'s matrix in block $C_i,C_i$ and $0$ outside this block. We denote this as $v_{ii}$. 

    Then for all $\beta\neq (0,0,\cdots, 0)$ we have $\langle W_i(\alpha),W_i(\beta)\rangle=0$, so $\langle v_{ii},W_i(\beta)\rangle=0$. Clearly $\langle v_{st},W_i(\beta)\rangle=0$ for all $s,t$ not both equal to $i$. Hence, for all $\beta\neq (0,0,\cdots, 0)$ we have $W_i(\beta)$ is orthogonal to $V$. Then $W_i(\beta)\in R$. As we chose $C_i$ arbitrarily, this is true for all $C_i\not\subseteq G'$.

    Now let $\mathcal{W}_i(\beta)$ be the ideal generated by $W_i(\beta)$. Then for all $\beta\neq (0,0,\cdots, 0)$, $\mathcal{W}_i(\beta)\subseteq R$. There are a total of $p^k-1$ choices for $\beta\neq (0,0,\cdots, 0)$. From the proof of Proposition \ref{cor:p2dim} there are $p^{n-k}-1$  classes $C_i\not\subseteq G'$. Hence, we have found $(p^k-1)(p^{n-k}-1)$ orthogonal ideals all in $R$. This equals the dimension of $R$, so each $\mathcal{W}_i(\beta)$ must be $1-$dimensional and thus irreducible. Decomposing $R$ into its irreducible components we have
    \[T(G)=V\oplus \bigoplus_{i=1}^{p^{n-k}-1}\bigoplus_{\alpha}\mathcal{W}_i(\alpha).\qedhere\]
\end{proof}

\begin{Cor}\label{cor:2prim}
    Let $G$ have nilpotency class $2$ with $|Z(G)|=p^k$. Then the non-primary primitive central idempotents of $T(G)$ are $W_i(\alpha)$ for any $\alpha=(\alpha_1,\alpha_2,\cdots, \alpha_k)$ where each $\alpha_j\in \{0,1,\cdots, p-1\}$ and not all $\alpha_j$ are $0$.
\end{Cor}

As a result of both $G/G'$ and $G'/Z(G)$ being elementary abelian groups for the nilpotency class $3$ case (Proposition \ref{thm:Camina3}) many of the arguments used for the nilpotency class $2$ case will also work in the nilpotency class $3$ case. We start with lemmas to aid in constructing the idempotents.

\begin{Lem}\label{lem:nil3stuctrure1}
    Suppose that $|Z(G)|=p^k$ for some $k>0$ and $|G'\colon Z(G)|=p^n$. Suppose that $Z(G)=\langle z_1, z_2,\cdots, z_k\rangle$. Suppose that $G$ has nilpotency class $3$ and let $C_i\subseteq G\setminus G'$ be a class. Let $C_{m,j}=\{z_m^j\}$ be one of the central classes of $G$. Then $(E_i^*A_{m,1}E_i^*)^j=E_i^*A_{m,j}E_i^*$ for all $j\in \{0,1,2,\cdots p-1\}$.
\end{Lem}

\begin{proof}
    The proof is similar to that of  Lemma \ref{lem:nil2stuctrure1} using Lemma \ref{lem:cam3outer} instead of Lemma \ref{lem:cam3centeral}.
\end{proof}

\begin{Lem}\label{lem:nil3structure2}
    Let $G$ have nilpotency class $3$. Suppose $|Z(G)|=p^k$ for some $k>0$ and $|G'\colon Z(G)|=p^n$. Suppose $Z(G)=\langle z_1, z_2,\cdots, z_k\rangle$. Let $C_i\subseteq G\setminus G'$ be a class. Suppose for $1\leq j\leq k$ we have $C_j=\{z_j\}$. 
    Let $C_\ell=\{z_1^{m_1}z_2^{m_2}\cdots z_k^{m_k}\}$ be any central class of $G$. Then $E_i^*A_\ell E_i^*=\prod_{j=1}^k (E_i^*A_{j}E_i^*)^{m_j}$.
\end{Lem}

\begin{proof}
    Similar to the proof of Lemma \ref{lem:nil2structure2} replacing Lemmas \ref{lem:cam3centeral} and \ref{lem:nil2stuctrure1} by Lemmas \ref{lem:cam3outer} and \ref{lem:nil3stuctrure1}.
\end{proof}

\begin{Prop}\label{prop:kroneckerdec3}
    Let $G$ have nilpotency class $3$. Suppose $|Z(G)|=p^k$ for some $k>0$ and $|G'\colon Z(G)|=p^n$. Suppose that $Z(G)=\langle z_1, z_2,\cdots, z_k\rangle$. Let $C_j\subseteq G\setminus G'$. Let $C_i=\{z_i\}$. Then the $C_j,C_j$ block of $A_i$ can be written as $\bigotimes_{m=1}^{n+k} M_m$ where each $M_m$ is a $p\times p$ matrix with $M_m=I$ for $m\neq n+k-i+1$ and
    \begin{equation}\label{eq:tensor3}
    M_{n+k-i+1}=\begin{pmatrix}
        0 & 1 & 0 & \cdots & 0 \\
        0 & 0 & 1 & \cdots & 0 \\
        \vdots & \vdots & & \ddots & \vdots \\
        0 & 0 & 0 & \cdots & 1 \\
        1 & 0 & 0 & \cdots & 0
    \end{pmatrix}\end{equation}
    for a specific ordering of the elements $G'$.
\end{Prop}

\begin{proof}
    Let $A_i'$ denote the $C_j,C_j$ block of $A_i$ and say $C_j=xG'$. We order the elements of $Z(G)$ as follows. First let $\Lambda$ be the sequence $1,z_1,z_1^2,\cdots, z_1^{p-1}$. We define $\Lambda(m_2,m_3,\cdots, m_k)=z_k^{m_k}z_{k-1}^{m_{k-1}}\cdots z_2^{m_2}\cdot\Lambda$. Now we order the elements of $Z(G)$ as \\ $\Lambda, \Lambda(1,0,\cdots, 0),\Lambda(2,0,\cdots, 0), \cdots, \Lambda(p-1,0,\cdots, 0), \Lambda(0,1,\cdots, 0), \Lambda(1,1,\cdots, 0),\cdots, \Lambda(p-1,1,\cdots, 0),\\ \Lambda(0,2,\cdots 0), \Lambda(1,2,\cdots 0),\cdots, \Lambda(p-1,p-1,0\cdots, 0),\Lambda(0,0,1,\cdots 0),\Lambda(1,0,1,\cdots, 0),\\ \cdots,  \Lambda(p-1,p-1,\cdots, p-1)$. Let $\Sigma$ denote this sequence of elements of $Z(G)$.

    By Theorem \ref{thm:Camina3}, $G'/Z(G)$ is elementary abelian, say $G'/Z(G)=\langle \overline{g_1},\overline{g_2},\cdots, \overline{g_n}\rangle$, where each $g_i\in G$ is a representative in $G'$ of a generator of $G'/Z(G)$. First let $\Upsilon$ be the sequence $1,g_1,g_1^2,\cdots, g_1^{p-1}$. We define $\Upsilon(m_2,m_3,\cdots, m_n)=g_n^{m_n}g_{n-1}^{m_{n-1}}\cdots g_2^{m_2}\cdot\Upsilon$. Now we order the representatives of $G'/Z(G)$ as $\Upsilon, \Upsilon(1,0,\cdots, 0),\Upsilon(2,0,\cdots, 0), \cdots, \Upsilon(p-1,0,\cdots, 0), \Upsilon(0,1,\cdots, 0), \Upsilon(1,1,\cdots, 0),\cdots, \Upsilon(p-1,1,\cdots,0), \\ \Upsilon(0,2,\cdots 0),\Upsilon(1,2,\cdots 0),\cdots,\Upsilon(p-1,p-1,0\cdots, 0),\Upsilon(0,0,1,\cdots 0),\Upsilon(1,0,1,\cdots, 0),\cdots, \\ \Upsilon(p-1,p-1,\cdots, p-1)$. Let $\upsilon_i$ be the $i$th element of this sequence of elements of $G'$. We now order $G'$ as $\upsilon_1\Sigma, \upsilon_2 \Sigma,\cdots, \upsilon_{p^n}\Sigma$.  

    With this ordering the rest of the proof is very similar to the proof of Proposition \ref{prop:kroneckerdec2} using Lemmas \ref{lem:cam3outer} and \ref{lem:nil3structure2} instead of Lemmas \ref{lem:cam3centeral} and \ref{lem:nil2structure2}.
    \end{proof}

    We now have the tools needed to construct a second type of idempotent for the nilpotency class $3$ case. We later prove they are orthogonal to one another.

    \begin{Thm}\label{thm:3idm}
   Let $G$ have nilpotency class $3$. Suppose $|Z(G)|=p^k$, $k>0$, and $|G'\colon Z(G)|=p^n$. Suppose that $Z(G)=\langle z_1, z_2,\cdots, z_k\rangle$. Let $C_j=\{z_j\}$ be the class containing $z_j$ for any $1\leq j\leq k$. Let $C_i\subseteq G\setminus G'$ be a class. Let $\zeta$ be any primitive $p$th root of unity. Then
    \[X_i(\alpha)=\frac{1}{p^k}\sum_{a_1,a_2,\cdots, a_k\in \{0,1,\cdots, p-1\}} \prod_{j=1}^k \zeta^{\alpha_j a_j}(E_i^*A_jE_i^*)^{a_j}\]
    is an idempotent of $T(G)$ for any $\alpha=(\alpha_1,\alpha_2,\cdots, \alpha_k)$ with each $\alpha_j\in \{0,1,\cdots, p-1\}$.
\end{Thm}

\begin{proof}
    This is similar to the proof of Theorem \ref{thm:2idm} using Proposition \ref{prop:kroneckerdec3} instead of Proposition \ref{prop:kroneckerdec2}.
\end{proof}

\begin{Prop}\label{prop:orth3}
    Let $G$ have nilpotency class $3$. Suppose $|Z(G)|=p^k$, $k>0$, and $|G'\colon Z(G)|=p^n$. Suppose that $Z(G)=\langle z_1, z_2,\cdots, z_k\rangle$. Let $C_j=\{z_j\}$ be the class containing $z_j$ for $1\leq j\leq k$. Let $C_i\subseteq G\setminus G'$ be a class and let $\zeta$ be any primitive $p$th root of unity. Then using the notation from Theorem \ref{thm:3idm} if $\alpha\neq \beta$, then $X_i(\alpha)$ is orthogonal to $X_i(\beta)$.
\end{Prop}

\begin{proof}
    The proof of this Proposition is similar to that of Proposition \ref{prop:orth2} using Theorem \ref{thm:3idm} in place of Theorem \ref{thm:2idm}.
\end{proof}

Now we now prove that these new idempotents we found are orthogonal to $V$.

\begin{Prop}\label{prop:orth3v}
    Let $G$ have nilpotency class $3$. Suppose $|Z(G)|=p^k$, $k>0$, and $|G'\colon Z(G)|=p^n$. Suppose that $Z(G)=\langle z_1, z_2,\cdots, z_k\rangle$ and $G'/Z(G)=\langle g_1Z(G), g_2Z(G),\cdots, g_nZ(G)\rangle$. Let $C_i\subseteq G\setminus G'$ be a class. For $1\leq j\leq k$ we let $C_{j}=\{z_j\}$ be a central class of $G$. Let $\zeta$ be any primitive $p$th root of unity. Let $X_i(\beta)$ be any idempotent of $T(G)$ from Theorem \ref{thm:3idm} for any $\beta=(\beta_1,\beta_2,\cdots \beta_k)$ with $\beta_j\in \{0,1,\cdots, p-1\}$ not all $0$. Then $X_i(\beta)$ is orthogonal to $V$, where $V$ is the primary component of $T(G)$.
\end{Prop}

\begin{proof}
    For any $\beta=(\beta_1,\beta_2,\cdots, \beta_k)$ we have
    \[X_i(\beta)=\frac{1}{p^k}\sum_{a_1,a_2,\cdots, a_k\in \{0,1,\cdots, p-1\}} \prod_{j=1}^k \zeta^{\beta_j a_j}(E_i^*A_jE_i^*)^{a_j}.\]
    Since $X_i(\beta)$ is $0$ outside of the $C_i,C_i$ block, the inner product of $X_i(\beta)$ with any matrix is the same as if we only considered the $C_i,C_i$ block of each. For this reason we set
    \[X_i'(\beta)=\frac{1}{p^k}\sum_{a_1,a_2,\cdots, a_k\in \{0,1,\cdots, p-1\}} \prod_{j=1}^k \zeta^{\beta_j a_j}(A_j')^{a_j}\]
    where $A_j'$ is the $C_i,C_i$ block of the matrix $A_j$. 
    Let $P$ be the $p\times p$ matrix in (\ref{eq:tensor3}).
     Then using an argument similar to that used in the proof of Theorem \ref{thm:2idm} we have 
    \[X_i'(\beta)=\frac{1}{p^k}\sum_{a_1,a_2,\cdots, a_k\in \{0,1,\cdots, p-1\}} \bigotimes_{i=1}^n I_p\otimes \zeta^{\alpha_j a_j} \bigotimes_{j=1}^{k} P^{a_j}=\frac{1}{p^k}I_{p^n}\otimes \sum_{a_1,a_2,\cdots, a_k}\zeta^{\alpha_j a_j} \bigotimes_{j=1}^{k} P^{a_j}.\]

    Let us now consider $X_i'(\alpha)$ when $\alpha=(0,0,\cdots, 0)$. We have 
    \[X_i'(\alpha)=\frac{1}{p^k}\sum_{a_1,a_2,\cdots, a_k\in \{0,1,\cdots, p-1\}} \bigotimes_{i=1}^n I_p\otimes \bigotimes_{j=1}^{k} P^{a_j}=\frac{1}{p^k}I_{p^n}\otimes \sum_{a_1,a_2,\cdots, a_k} \bigotimes_{j=1}^k P^{a_j}.\]
    We claim that $\sum_{a_1,a_2,\cdots, a_k} \bigotimes_{j=1}^k P^{a_j}=J_{p^k}$. We proceed by induction on $k\geq 1$. For $k=1$,
    \[\sum_{a_1,a_2,\cdots, a_k} \bigotimes_{j=1}^k P^{a_j}=\sum_{i=0}^{p-1}P^i=J_p\]
    so the result is true for $k=1$. Suppose the result is true for some $m\geq 1$. That is, $\sum_{a_1,a_2,\cdots, a_m} \bigotimes_{j=1}^m P^{a_j}=J_{p^m}$. We consider $m+1$. We have
    \[\sum_{a_1,a_2,\cdots, a_{m+1}} \bigotimes_{j=1}^{m+1} P^{a_j}=\sum_{a_1=0}^{p-1}\sum_{a_2,a_3,\cdots, a_{m+1}} \bigotimes_{j=1}^{m+1}P^{a_j}=\sum_{a_1=0}^{p-1}P^{a_1}\otimes \sum_{a_2,a_3,\cdots, a_{m+1}}\bigotimes_{j=2}^{m+1}P^{a_j}.\]
    The first term is a sum of all the powers of $P$, and is thus $J_p$. The second term, by the inductive hypothesis, is $J_{p^m}$. Then this equals $J_p\otimes J_{p^m}=J_{p^{m+1}}$.
    This proves the inductive step. Hence, $\sum_{a_1,a_2,\cdots, a_k} \bigotimes_{j=1}^k P^{a_j}=J_{p^k}$. We then have $X_i'(\alpha)=\frac{1}{p^k}I_{p^n}\otimes J_{p^k}$.

    Now for any $\beta=(\beta_1,\beta_2,\cdots, \beta_k)$, with not all the $\beta_j=0$, we have by Proposition \ref{prop:orth3}, $\langle X_i'(\beta),X_i'(\alpha)\rangle=0$. Thus,
    \[0=\langle X_i'(\beta),X_i'(\alpha)\rangle=\trace\left(\left(\frac{1}{p^k}I_{p^n}\otimes \sum_{a_1,a_2,\cdots, a_k}\zeta^{\beta_j a_j} \bigotimes_{j=1}^{k} P^{a_j}\right)\left(\overline{\frac{1}{p^k}I_{p^n}\otimes J_{p^k}}\right)^T\right)\]
    \[=\trace\left(\left(\frac{1}{p^k}I_{p^n}\otimes \sum_{a_1,a_2,\cdots, a_k}\zeta^{\beta_j a_j} \bigotimes_{j=1}^{k} P^{a_j}\right)\left(\frac{1}{p^k}I_{p^n}\otimes J_{p^k}\right)\right)=\trace\left(\frac{1}{p^{2k}}I_{p^n}\otimes\left(\sum_{a_1,a_2,\cdots, a_k}\zeta^{\beta_j a_j} \bigotimes_{j=1}^{k} P^{a_j}\right)J_{p^k}\right).\]
    We note since the first term in the Kronecker product is $I_{p^n}$ this trace is just the trace of \\ $\left(\sum_{a_1,a_2,\cdots, a_k}\zeta^{\beta_j a_j} \bigotimes_{j=1}^{k} P^{a_j}\right)J_{p^k}$ multiplied by $\frac{p^n}{p^{2k}}$ since in the trace computation we just add up the trace of $\left(\sum_{a_1,a_2,\cdots, a_k}\zeta^{\beta_j a_j} \bigotimes_{j=1}^{k} P^{a_j}\right)J_{p^k}$ a total $p^n$ times. We then have
    \[0=\frac{p^n}{p^{2k}}\trace\left(\left(\sum_{a_1,a_2,\cdots, a_k}\zeta^{\beta_j a_j} \bigotimes_{j=1}^{k} P^{a_j}\right)J_{p^k}\right).\]
    So $\trace\left(\left(\sum_{a_1,a_2,\cdots, a_k}\zeta^{\beta_j a_j} \bigotimes_{j=1}^{k} P^{a_j}\right)J_{p^k}\right)=0$.

    We now let $v_{ii}$ be the basis element of $V$ with all $1$'s in the $C_i,C_i$ block. We let $v'_{ii}$ be the $C_i,C_i$ block of $v_{ii}$. We have $v'_{ii}=J_{p^{n+k}}=J_{p^n}\otimes J_{p^k}$. Observe
    \[\langle X_i'(\beta),v'_{ii}\rangle=\trace\left(\left(\frac{1}{p^k}I_{p^n}\otimes \sum_{a_1,a_2,\cdots, a_k}\zeta^{\beta_j a_j} \bigotimes_{j=1}^{k} P^{a_j}\right)\left(\overline{J_{p^n}\otimes J_{p^k}}\right)^T\right)\]
    \[=\trace\left(\left(\frac{1}{p^k}I_{p^n}\otimes \sum_{a_1,a_2,\cdots, a_k}\zeta^{\beta_j a_j} \bigotimes_{j=1}^{k} P^{a_j}\right)\left(J_{p^n}\otimes J_{p^k}\right)\right)=\trace\left(\frac{1}{p^{k}}J_{p^n}\otimes\left(\sum_{a_1,a_2,\cdots, a_k}\zeta^{\beta_j a_j} \bigotimes_{j=1}^{k} P^{a_j}\right)J_{p^k}\right).\]
    Since the first term in the Kronecker product has all $\frac{1}{p^k}$'s along its main diagonal, the trace we are computing is $\frac{p^n}{p^k}\trace\left(\left(\sum_{a_1,a_2,\cdots, a_k}\zeta^{\beta_j a_j} \bigotimes_{j=1}^{k} P^{a_j}\right)J_{p^k}\right)=0$. Hence, $\langle X_i'(\beta),v'_{ii}\rangle=0$, and so $\langle X_i(\beta),v_{ii}\rangle=0$. For any $s,t$ not both $i$ we have $\langle X_i(\beta),v_{st}\rangle=0$ as $X_i(\beta)$ is only nonzero in the $C_i,C_i$ block and $v_{st}$ is all $0$ in the $C_i,C_i$ block. Hence, for any $\beta=(\beta_1,\beta_2,\cdots \beta_k)$ with $\beta_j\in \{0,1,\cdots, p-1\}$ not all $0$, $X_i(\beta)$ is orthogonal to $V$. 
\end{proof}

We now consider the third type of idempotent for $T(G)$ using a similar argument to the first two. 

\begin{Lem}\label{lem:nil3stuctrure1'}
    Let $G$ have nilpotency class $3$. Suppose $|Z(G)|=p^k$, $k>0$, and $|G'\colon Z(G)|=p^n$. Suppose that $G'/Z(G)=\langle g_1Z(G), g_2Z(G),\cdots, g_nZ(G)\rangle$. Let $C_i\subseteq G\setminus G'$ be a class. Let $C_{m,j}=g_m^jZ(G)$ be one of the classes of $G$. Then $(E_i^*A_{m,1}E_i^*)^j=p^{(j-1)k}E_i^*A_{m,j}E_i^*$ for all $j\in \{1,2,\cdots p-1\}$. Furthermore $(E_i^*A_{m,0}E_i^*)^0=E_i^*A_{m,0}E_i^*$.
\end{Lem}

\begin{proof}
    The statement $(E_i^*A_{m,1}E_i^*)^0=E_i^*A_{m,0}E_i^*$ is true as both are the identity in the $C_i,C_i$ block and $0$ outside this block. We now consider the first part. From the proof of Lemma \ref{lem:cam3outer} we have, for all $x,y\in C_i$, $(E_i^*A_{m,j} E_i^*)_{xy}=1$ if and only if $y\in xC_{m,j}=xg_m^jZ(G)$. We proceed by induction on $j$. For $j=1$, the result is clear. We now suppose the result is true for some $\ell\geq 1$, with $\ell\leq p-2$. Then $(E_i^*A_{m,1}E_i^*)^\ell=p^{(\ell-1)k}E_i^*A_{m,\ell}E_i^*$, and so  $(E_i^*A_{m,1}E_i^*)^{\ell+1}=(E_i^*A_{m,1}E_i^*)^\ell(E_i^*A_{m,1}E_i^*)=p^{(\ell-1)k}(E_i^*A_{m,\ell}E_i^*)(E_i^*A_{m,1}E_i^*)$.
    We observe
    \[[(E_i^*A_{m,\ell}E_i^*)(E_i^*A_{m,1}E_i^*)]_{xy}=\sum_{u\in G}(E_i^*A_{m,\ell}E_i^*)_{xu}(E_i^*A_{m,1}E_i^*)_{uy}.\]
    From Lemma \ref{lem:cam3outer}, $(E_i^*A_{m,\ell}E_i^*)_{xu}(E_i^*A_{m,1}E_i^*)_{uy}=1$ if and only if $u\in xg_m^\ell Z(G)$ and $y\in ug_m Z(G)=xg_m^{\ell+1}Z(G)$.
    There are $p^k$ different choices for $u$, namely those in $xg_m^\ell Z(G)$, that can make the first term in each product nonzero. For each of these choices of $u$, if $y\not\in ug_m^{\ell}Z(G)$ then $(E_i^*A_{m,\ell}E_i^*)_{xu}(E_i^*A_{m,1}E_i^*)_{uy}=0$ and if $y\in ug_m^\ell Z(G)$, then $(E_i^*A_{m,\ell}E_i^*)_{xu}(E_i^*A_{m,1}E_i^*)_{uy}=1$. Thus, if $y\in xg_{m}^{\ell+1}Z(G)$, we have a nonzero term in our sum for each of the $u\in xg_m^\ell Z(G)$. This gives a total of $p^k$ nonzero terms in the sum, each of which equals 1. Thus, if $y\in xg_m^{\ell+1}Z(G)$, the $(x,y)$ entry of $(E_i^*A_{m,\ell}E_i^*)(E_i^*A_{m,1}E_i^*)$ is $p^k$. If $y\not\in xg_{m}^{\ell+1}Z(G)$, then $(E_i^*A_{m,\ell}E_i^*)_{xu}(E_i^*A_{m,1}E_i^*)_{uy}=0$ for all $u$ as either $u\not\in g_m^\ell Z(G)$, or $y\not\in ug_m^\ell Z(G)$. Thus, the $(x,y)$ entry of  $(E_i^*A_{m,\ell}E_i^*)(E_i^*A_{m,1}E_i^*)$ is $0$ in this case.
    
    Hence, the $(x,y)$ entry of $(E_i^*A_{m,\ell}E_i^*)(E_i^*A_{m,1}E_i^*)$ is $p^k$ if and only if $y\in xz_m^{\ell+1}Z(G)$. From the proof of Lemma \ref{lem:cam3outer} this is identical to the condition for the $(x,y)$ entry of $p^k E_i^*A_{m,\ell+1}E_i^*$ to be $p^k$. Hence, $[(E_i^*A_{m,\ell}E_i^*)(E_i^*A_{m,1}E_i^*)]_{xy}=p^k(E_i^*A_{m,\ell+1}E_i^*)_{xy}$ for all $x,y\in C_{i}$. Thus, \[E_i^*A_{\beta}E_i^*=(E_i^*A_{m,1}E_i^*)^{\ell+1}=p^{(\ell-1)k}(E_i^*A_{m\ell}E_i^*)(E_i^*A_{m,1}E_i^*)=p^{\ell k}E_{i}^*A_{m,\ell+1}E_i^*.\] Then by induction we have $(E_i^*A_{m,1}E_i^*)^j=p^{(j-1)k}E_i^*A_{m,j}E_i^*$ for all $j\in \{1,\cdots p-1\}$. 
\end{proof}

\begin{Lem}\label{lem:nil3structure2'}
   Let $G$ have nilpotency class $3$. Suppose $|Z(G)|=p^k$, $k>0$, $|G'\colon Z(G)|=p^n$, and $G'/Z(G)=\langle g_1Z(G), g_2Z(G),\cdots, g_nZ(G)\rangle$. Let $C_i\subseteq G\setminus G'$ be a class.
    Let $C_\ell=g_1^{m_1}g_2^{m_2}\cdots g_n^{m_n}Z(G)$ be any conjugacy class of $G$ $(0\leq m_i< p)$. Then $E_i^*A_\ell E_i^*=p^k\prod_{j=1}^n \frac{1}{p^{m_j k}} (E_i^*A_{j}E_i^*)^{m_j}$.
\end{Lem}

\begin{proof}
    Let $L$ be the number of nonzero $m_j$ in $C_\ell=g_1^{m_1}g_2^{m_2}\cdots g_n^{m_n}Z(G)$. We proceed by induction on $L$. For $L=1$ we assume, without loss of generality, that $m_1$ is the only nonzero $m_j$. Then 
    \[\prod_{j=1}^n \frac{1}{p^{\gamma_j k}}(E_i^*A_{j}E_i^*)^{m_j}=\frac{p^k}{p^{m_1k}}(E_i^*A_jE_i^*)^{m_1}=\frac{1}{p^{(m_1-1)k}}(E_i^*A_{j}E_i^*)^{m_1}.\]
    By Lemma \ref{lem:nil3stuctrure1'}, this equals $E_i^*A_\ell E_i^*$. So the result is true for $L=1$. Now assume the result is true for some $1\leq L\leq n-1$. That is, for all $C_\ell=g_1^{m_1}g_2^{m_2}\cdots g_n^{m_n}Z(G)$ with $L$ of the $m_j\neq 0$ we have
    \[E_i^*A_\ell E_i^*=p^k\prod_{j=1}^n \frac{1}{p^{m_j k}} (E_i^*A_{j}E_i^*)^{m_j}.\]
    Now we consider $C_\alpha=g_1^{a_1}g_2^{a_2}\cdots g_n^{a_n}Z(G)$ in which $L+1$ of the $a_j\neq 0$. Without loss of generality, we may assume the first $L+1$ of the $a_j\neq 0$. Then we have
    \[p^k\prod_{j=1}^n \frac{1}{p^{a_j k}}(E_i^*A_{j}E_i^*)^{a_j}=p^k\prod_{j=1}^{L+1}\frac{1}{p^{a_j k}} (E_i^*A_jE_i^*)^{a_j}=\left(p^k\prod_{j=1}^{L} \frac{1}{p^{a_j k}}(E_i^*A_jE_i^*)^{a_j}\right)\frac{1}{p^{a_{L+1}k}}(E_i^*A_{L+1}E_i^*)^{a_{L+1}}.\]
    Letting $C_\beta=g_1^{a_1}g_2^{a_2}\cdots g_L^{a_L}Z(G)$ we have by the inductive hypothesis that
    \[\left(p^k\prod_{j=1}^{L} \frac{1}{p^{a_j}}(E_i^*A_jE_i^*)^{a_j}\right)\frac{1}{p^{a_{L+1}k}}(E_i^*A_{L+1}E_i^*)^{a_{L+1}}=(E_i^*A_\beta E_i^*)\frac{1}{p^{a_{L+1}k}}(E_i^*A_{L+1}E_i^*)^{a_{L+1}}.\]
    Letting $C_\upsilon=g_{L+1}^{a_{L+1}}Z(G)$ by Lemma \ref{lem:nil3stuctrure1'}, $(E_i^*A_{L+1}E_i^*)^{a_{L+1}}=p^{(a_{L+1}-1)k}E_i^*A_\upsilon E_i^*$. We then have
    \[p^k\prod_{j=1}^n \frac{1}{p^{a_j k}}(E_i^*A_{j}E_i^*)^{a_j}=(E_i^*A_\beta E_i^*)\frac{1}{p^k}(E_i^*A_\upsilon E_i^*)=\frac{1}{p^k}(E_i^*A_\beta E_i^*)(E_i^*A_\upsilon E_i^*).\]
    We observe that $[(E_i^*A_{\beta}E_i^*)(E_i^*A_{\upsilon}E_i^*)]_{xy}=\sum_{u\in G}(E_i^*A_{\beta}E_i^*)_{xu}(E_i^*A_{\upsilon}E_i^*)_{uy}$.
    
     From Lemma \ref{lem:cam3outer}, we have $(E_i^*A_{\beta}E_i^*)_{xu}(E_i^*A_{\upsilon}E_i^*)_{uy}=1$ if and only if $u\in xC_\beta$ and $y\in u C_\upsilon=xg_1^{a_1}g_2^{a_2}\cdots g_{L+1}^{a_{L+1}}Z(G)=xC_\alpha$.
    The first of these conditions states there are $p^k$ different choices for $u$, namely those in $xC_\beta$ that can make the first term in each product nonzero. For each of these choices of $u$ if $y\not\in xC_\alpha$ then, $(E_i^*A_{\beta}E_i^*)_{xu}(E_i^*A_{\upsilon}E_i^*)_{uy}=0$ and if $y\in xC_\alpha$, then $(E_i^*A_{\beta}E_i^*)_{xu}(E_i^*A_{\upsilon}E_i^*)_{uy}=1$. Thus, if $y\in xC_\alpha$, we have a nonzero term in our sum for each of the $u\in xC_\beta$. This gives a total of $p^k$ nonzero terms in the sum, each of which is $1$. Thus, if $y\in xC_\alpha$, the $(x,y)$ entry of $(E_i^*A_{\beta}E_i^*)(E_i^*A_{\upsilon}E_i^*)$ is $p^k$. If $y\not\in xC_\alpha$, then $(E_i^*A_{\beta}E_i^*)_{xu}(E_i^*A_{\upsilon}E_i^*)_{uy}=0$ for all $u$ as either $u\not\in xC_\beta$, or $y\not\in uC_\upsilon$. Thus, the $(x,y)$ entry of  $(E_i^*A_{\beta}E_i^*)(E_i^*A_{\upsilon}E_i^*)$ is $0$ in this case.
    
    Hence, the $(x,y)$ entry of $(E_i^*A_{\beta}E_i^*)(E_i^*A_{\upsilon}E_i^*)$ is $p^k$ if and only if $y\in xC_\alpha$. From the proof of Lemma \ref{lem:cam3outer} this is the identical condition for the $(x,y)$ entry of $p^k E_i^*A_{C_\alpha}E_i^*$ to be $p^k$. Hence, we have $[(E_i^*A_{\beta}E_i^*)(E_i^*A_{\upsilon}E_i^*)]_{xy}=p^k(E_i^*A_{\alpha}E_i^*)_{xy}$ for all $x,y\in C_{i}$. We note that $\prod_{j=1}^n \frac{1}{p^{a_j k}}(E_i^*A_{j}E_i^*)^{a_j}$ is $0$ outside of the $C_i,C_i$, block as multiplying by $E_i^*$ on the right and left makes everything outside of this block $0$. Hence, 
    \[p^k\prod_{j=1}^n \frac{1}{p^{a_j k}}(E_i^*A_{j}E_i^*)^{a_j}=\frac{1}{p^k}(E_i^*A_{\beta}E_i^*)(E_i^*A_{\upsilon}E_i^*)= E_i^*A_\alpha E_i^*.\]
    This proves the inductive step and concludes the proof.
\end{proof}

\begin{Prop}\label{prop:kroneckerdec3'}
    Let $G$ have nilpotency class $3$. Suppose $|Z(G)|=p^k$, $k>0$, and $|G'\colon Z(G)|=p^n$. Suppose that $Z(G)=\langle z_1, z_2,\cdots, z_k\rangle$ and $G'/Z(G)=\langle g_1Z(G), g_2Z(G),\cdots, g_nZ(G)\rangle$. Let $C_j\subseteq G\setminus G'$.
    Let $C_\ell=g_1^{m_1}g_2^{m_2}\cdots g_n^{m_n}Z(G)$ be any class of $G$. Then we have that the $C_j,C_j$ block of $A_\ell$ can be written as $\left(\bigotimes_{t=1}^{n} P^{m_t}\right)\otimes J_{p^k}$ where $P$ is the $p\times p$ matrix
    \begin{equation}\label{eq:tensor3'}
    P=\begin{pmatrix}
        0 & 1 & 0 & \cdots & 0 \\
        0 & 0 & 1 & \cdots & 0 \\
        \vdots & \vdots & & \ddots & \vdots \\
        0 & 0 & 0 & \cdots & 1 \\
        1 & 0 & 0 & \cdots & 0
    \end{pmatrix}\end{equation}
    for a specific ordering of the elements $G'$.
\end{Prop}

\begin{proof}
    Let $A_\ell'$ denote the $C_j,C_j$ block of $A_\ell$, where $C_j=gG'$. We shall order the elements of $Z(G)$ as follows. Let $\Lambda$ be the sequence $1,z_1,z_1^2,\cdots, z_1^{p-1}$. We define $\Lambda(m_2,m_3,\cdots, m_k)=z_k^{m_k}z_{k-1}^{m_{k-1}}\cdots z_2^{m_2}\cdot\Lambda$. Now we order the elements of $Z(G)$ as $\Lambda, \Lambda(1,0,\cdots, 0),\Lambda(2,0,\cdots, 0), \cdots, \Lambda(p-1,0,\cdots, 0), \\ \Lambda(0,1,\cdots, 0), \Lambda(1,1,\cdots, 0),\cdots, \Lambda(p-1,1,\cdots, 0),\Lambda(0,2,\cdots 0), \Lambda(1,2,\cdots 0),\cdots, \Lambda(p-1,p-1,0\cdots, 0), \\ \Lambda(0,0,1,\cdots 0),\Lambda(1,0,1,\cdots, 0),\cdots,  \Lambda(p-1,p-1,\cdots, p-1)$. Let $\Sigma$ denote this sequence of elements of $Z(G)$.

   Now let $\Upsilon$ be the sequence $1,g_1,g_1^2,\cdots, g_1^{p-1}$. We define $\Upsilon(m_2,m_3,\cdots, m_n)=g_n^{m_n}g_{n-1}^{m_{n-1}}\cdots g_2^{m_2}\cdot\Upsilon$. Now we order the representatives of $G'/Z(G)$ as $\Upsilon, \Upsilon(1,0,\cdots, 0),\Upsilon(2,0,\cdots, 0), \cdots, \Upsilon(p-1,0,\cdots, 0), \\ \Upsilon(0,1,\cdots, 0), \Upsilon(1,1,\cdots, 0),\cdots, \Upsilon(p-1,1,\cdots,0), \Upsilon(0,2,\cdots 0),\Upsilon(1,2,\cdots 0),\cdots,\Upsilon(p-1,p-1,0\cdots, 0), \\ \Upsilon(0,0,1,\cdots 0),\Upsilon(1,0,1,\cdots, 0),\cdots, \Upsilon(p-1,p-1,\cdots, p-1)$. Let $\upsilon_i$ be the $i$th element of this sequence of elements of $G'$. We now order $G'$ as $\upsilon_1\Sigma, \upsilon_2 \Sigma,\cdots, \upsilon_{p^n}\Sigma$.

    From the proof of Lemma \ref{lem:cam3outer}, the $(x,y)$ entry of $A_\ell'$ is $1$ if and only if $y\in x C_\ell$. Using the ordering we have described above, suppose the $\upsilon_s Z(G), \upsilon_t Z(G)$ block of $A_\ell'$ has a nonzero entry. Say $(A_\ell')_{xy}=1$. Then $y\in x C_\ell$ with $x\in \upsilon_s Z(G)$ and $y\in \upsilon_t Z(G)$. Suppose $y=\upsilon_t w$ for some $w\in Z(G)$. Given any $a\in \upsilon_t Z(G)$ we have $a=\upsilon_t b$ for some $b\in Z(G)$. Then $a=\upsilon_t ww^{-1}b=y(w^{-1}b)$. As $w^{-1}b\in Z(G)$ and $y\in x C_\ell$, we have $a\in x C_\ell$ as $C_\ell$ is a coset of $Z(G)$ in $G'$. Hence, for all $a\in \upsilon_t Z(G)$ we have $a\in x C_\ell$. Therefore, $\upsilon_t Z(G) \subseteq x C_\ell= \upsilon_s C_\ell$. Now as $\upsilon_t Z(G) \subseteq \upsilon_t C_\ell$, for all $a\in \upsilon_s Z(G)$ and $b\in \upsilon_t Z(G)$, $b\in a C_\ell=\upsilon_s C_\ell$. Hence, $(A_\ell ')_{ab}=1$. Thus, every entry of the $\upsilon_s Z(G), \upsilon_t Z(G)$ block of $A_\ell'$ is $1$. So if any entry of the $\upsilon_s Z(G), \upsilon_t Z(G)$ block of $A_\ell'$ is nonzero, they are all nonzero. Hence, the $\upsilon_s Z(G), \upsilon_t Z(G)$ block of $A_\ell'$ is either all $1$'s or all $0$'s.

    Using Theorem \ref{thm:Camina3} the lower central series of $G$ is $G\trianglerighteq G'\trianglerighteq Z(G)\trianglerighteq 1$. Then the lower central series of $G/Z(G)$, which is a Camina group by Proposition \ref{prop:cammod}, is $G/Z(G)\trianglerighteq G'/Z(G)\trianglerighteq 1$. Hence, $G/Z(G)$ is a Camina group of nilpotency class $2$ with center $G'/Z(G)$. We note that $\overline{\upsilon_1},\overline{\upsilon_2}\cdots, \overline{\upsilon_{p^n}}$ is an ordering of the elements of $G'/Z(G)$. It is in fact the same ordering of the elements of the center of $G/Z(G)$ that is used in the proof of Proposition \ref{prop:kroneckerdec2}. For each conjugacy class of $G/Z(G)$, we let $\overline{C_i}$ denote the conjugacy class $C_i/Z(G)$ in $G/Z(G)$. In particular, $\overline{C_\ell}=\overline{g_1^{m_1}}\cdot\overline{g_2^{m_2}}\cdots \overline{g_{n}^{m_n}}$. Let $\overline{A_\ell}$ be the adjacency matrix corresponding to $\overline{C_\ell}$. We shall let $\overline{A_\ell'}$ denote the $\overline{C_j},\overline{C_j}$ block of $\overline{A_\ell}$. 
    
    From what we proved above, the $\upsilon_sZ(G),\upsilon_t Z(G)$ block of $A_\ell'$ is all $1$'s if and only if $\upsilon_tZ(G)\subseteq \upsilon_tC_\ell$. This is true if and only if $\overline{\upsilon_t}\subseteq \overline{\upsilon_t}\cdot\overline{C_\ell}$. By the proof of Lemma \ref{lem:cam3centeral} this is exactly the condition for the $(\overline{\upsilon_s},\overline{\upsilon_t})$ entry of $\overline{A_\ell'}$ to be nonzero. Hence, the $\upsilon_sZ(G),\upsilon_t Z(G)$ block of $A_\ell'$ is all $1$'s if and only if $(\overline{A_\ell'})_{\overline{\upsilon_s},\overline{\upsilon_t}}=1$. We now let $\overline{A_{r}}$ denote the adjacency matrix for the conjugacy class $\overline{C_r}=\{\overline{g_r}\}$ in $G/Z(G)$. By Lemma \ref{lem:nil2structure2}, $\overline{A_\ell'}=\prod_{r=1}^n \overline{A_r'}^{m_r}$, where $\overline{A_r'}$ is the $\overline{C_j},\overline{C_j}$ block of $\overline{A_r}$. As the ordering on $G'/Z(G)$ is the same as in the proof of Proposition \ref{prop:kroneckerdec2}, $\overline{A_r'}=\bigotimes_{a=1}^n M_a$ where each $M_a$ is a $p\times p$ matrix with $M_a=I$ for $m\neq n-r+1$ and $M_{n-r+1}=P$ from (\ref{eq:tensor3'}). 
    We observe that
    \[\overline{A_\ell'}=\prod_{r=1}^n \overline{A_r'}^{m_r}=\prod_{r=1}^n \bigotimes_{a=1}^n M_a^{m_r}=\bigotimes_{r=1}^n P^{m_r}\]
    since matrix multiplication distributes over Kronecker products.

    As the $\upsilon_sZ(G),\upsilon_t Z(G)$ block of $A_\ell'$ is all $1$'s if and only if $(\overline{A_\ell'})_{\overline{\upsilon_s},\overline{\upsilon_t}}=1$, we have $A_\ell'=\overline{A_\ell'}\otimes J$ since each place $\overline{A_\ell'}$ has a $1$, $A_\ell'$ has a $J$ matrix appearing. Therefore, we have $A_\ell'=\left(\bigotimes_{t=1}^n P^{m_t}\right)\otimes J$.
\end{proof}

We will need one more result before constructing the idempotents.

\begin{Prop}\label{prop:kroneckerdec3c}
     Let $G$ have nilpotency class $3$. Suppose $|Z(G)|=p^k$, $k>0$, and $|G'\colon Z(G)|=p^n$. Suppose $Z(G)=\langle z_1, z_2,\cdots, z_k\rangle$ and $G'/Z(G)=\langle g_1Z(G), g_2Z(G),\cdots, g_nZ(G)\rangle$. Let $C_j\subseteq G\setminus G'$ be a class. Letting $C_0,C_1,\cdots, C_{p^k-1}$ be the central classes of $G$, the $C_j,C_j$ block of $\sum_{i=1}^{p^k}A_i$ can be written as $\left(\bigotimes_{t=1}^{n} I_{p}\right)\otimes J_{p^k}$.
\end{Prop}

\begin{proof}
    Let $A_i'$ denote the $C_j,C_j$ block of $A_i$ for $0\leq i\leq p^k-1$ and let $C_j=gG'$. We order the elements of $Z(G)$ and $G'$ as in the proof of Proposition \ref{prop:kroneckerdec3'}, using the same notation as in that proof.  

    From the proof of Lemma \ref{lem:cam3inner} the $(x,y)$ entry of $A_i'$ is $1$ if and only if $y\in xC_i=\{xz_i\}$. Then as $C_0,C_1,\cdots, C_{p^k-1}$ partition $Z(G)$, the $(x,y)$ entry of $\sum_{i=0}^{p^k-1} A_i'$ is $1$ if and only if $y\in xZ(G)$. To see this we note the $(x,y)$ entry of $\sum_{i=0}^{p^k-1}A_i'$ is $1$ if and only if $y\in xC_i$ for some $C_i$, if and only if $y\in xZ(G)$ as $Z(G)=\bigcup_{i=0}^{p^k-1} C_i$.

Using our ordering, suppose the $\upsilon_s Z(G), \upsilon_t Z(G)$ block of $\sum_{i=0}^{p^k-1} A_i'$ has a nonzero entry. Say $(\sum_{i=0}^{p^k-1} A_i')_{xy}=1$, where $x\in \upsilon_s Z(G)$ and $y\in \upsilon_t Z(G)$. Note that $y\in xZ(G)$ by the above argument. Suppose $y=\upsilon_t w$ for some $w\in Z(G)$. Given any $a\in \upsilon_t Z(G)$ we have that $a=\upsilon_t b$ for some $b\in Z(G)$. Then $a=\upsilon_t ww^{-1}b=y(w^{-1}b)$. As $w^{-1}b\in Z(G)$ and $y\in x Z(G)$, we have $a\in x Z(G)$. Hence, for all $a\in \upsilon_t Z(G)$ we have $a\in x Z(G)$. Therefore, $\upsilon_t Z(G) \subseteq x Z(G)= \upsilon_s Z(G)$. Now as $\upsilon_t Z(G) \subseteq \upsilon_s Z(G)$, we have for all $a\in \upsilon_s Z(G)$ and $b\in \upsilon_t Z(G)$, $b\in a Z(G)=\upsilon_s Z(G)$. Hence, $(\sum_{i=0}^{p^k-1} A_i')_{ab}=1$. Thus, every entry of the $\upsilon_s Z(G), \upsilon_t Z(G)$ block of $A_\ell'$ is $1$. So if any entry of the $\upsilon_s Z(G), \upsilon_t Z(G)$ block of $A_\ell'$ is nonzero, they are all nonzero. Hence, the $\upsilon_s Z(G), \upsilon_t Z(G)$ block of $A_\ell'$ is either all $1$'s or all $0$'s. We note $\upsilon_t Z(G)\subseteq \upsilon_s Z(G)$ if and only if $\upsilon_s=\upsilon_t$. Therefore, the $\upsilon_s Z(G),\upsilon_s Z(G)$ block of $\sum_{i=0}^{p^k-1} A_i'$ is all $1$'s for all $\upsilon_s$ and all other blocks are $0$. This however means $\sum_{i=0}^{p^k-1} A_i'=I_{p^{n}}\otimes J_{p^k}$. Therefore, $\sum_{i=0}^{p^k-1} A_i'=\left(\bigotimes_{t=1}^n I_p\right)\otimes J_{p^k}$.
\end{proof}

We now have the tools needed to construct the third type of idempotent for the nilpotency class $3$ type.

\begin{Thm}\label{thm:3'idm}
    Let $G$ have nilpotency class $3$. Suppose $|Z(G)|=p^k$, $k>0$, $|G'\colon Z(G)|=p^n$, $Z(G)=\langle z_1, z_2,\cdots, z_k\rangle$, and $G'/Z(G)=\langle g_1Z(G), g_2Z(G),\cdots, g_nZ(G)\rangle$. Let $C_i\subseteq G\setminus G'$ be a class. Let $C_0,C_1,\cdots, C_{p^k-1}$ be the central classes of $G$ and $C_{p^k-1+j}=g_jZ(G)$ for $1\leq j\leq n$. Let $\zeta$ be any primitive $p$th root of unity. Then define
    \[Y_i(\alpha)=\frac{1}{p^{n+k}}\left(\sum_{j=0}^{p^k-1} E_i^*A_jE_i^*+\sum_{a_1,a_2,\cdots, a_n} p^k\prod_{j=1}^n \frac{\zeta^{\alpha_j a_j}}{p^{a_jk}}(E_i^*A_{p^k-1+j}E_i^*)^{a_j}\right)\]
    where $a_1,a_2,\cdots, a_n\in \{0,1,\cdots, p-1\}$ and not all of the $a_j$ are zero. Then $Y_i(\alpha)$
    is an idempotent of $T(G)$ for any $\alpha=(\alpha_1,\alpha_2,\cdots, \alpha_n)$ with each $\alpha_j\in \{0,1,\cdots, p-1\}$.
\end{Thm}

\begin{proof}
     Fix $\alpha=(\alpha_1,\alpha_2,\cdots, \alpha_n)$ with each $\alpha_j\in \{0,1,\cdots, p-1\}$. We note $Y_i(\alpha)$ is only nonzero in the $C_i,C_i$ block. If we let $A_j'$ denote the $C_i,C_i$ block of $A_j$ the $C_i,C_i$ block of $Y_i(\alpha)$ is just 
    \[Y_i'(\alpha)=\frac{1}{p^{n+k}}\left(\sum_{j=0}^{p^k-1} (A_j')+\sum_{a_1,a_2,\cdots, a_n} p^k\prod_{j=1}^n \frac{\zeta^{\alpha_j a_j}}{p^{a_jk}}(A_{p^k-1+j}')^{a_j}\right).\]
    Hence, $Y_i(\alpha)$ is an idempotent if and only if $Y_i'(\alpha)$ is an idempotent.

    By Proposition \ref{prop:kroneckerdec3c}, $\sum_{j=0}^{p^k-1} (A_j')=\left(\bigotimes_{t=1}^n I_p\right)\otimes J_{p^k}$. Similarly be Proposition \ref{prop:kroneckerdec3'}, $A_{p^k-1+j}'=\left(\bigotimes_{t=1}^{n} P^{m_t}\right)\otimes J_{p^k}$ where $m_t=0$ for $t\neq j$, $m_j=1$, and $P$ is the $p\times p$ matrix in (\ref{eq:tensor3'}).

    Using this we have
    \[Y_i'(\alpha)=\frac{1}{p^{n+k}}\left(\left(\bigotimes_{t=1}^n I_p\right)\otimes J_{p^k}+\sum_{a_1,a_2,\cdots, a_n} p^k\prod_{j=1}^n \frac{\zeta^{\alpha_j a_j}}{p^{a_jk}} \left(\left(\bigotimes_{t=1}^n P^{m_t}\right)\otimes J_{p^k}\right)^{a_j}\right)\]
    \[=\frac{1}{p^{n+k}}\left(\left(\bigotimes_{t=1}^n I\right)\otimes J_{p^k}+\sum_{a_1,a_2,\cdots, a_n} p^k\prod_{j=1}^n \zeta^{\alpha_j a_j}\left(\bigotimes_{t=1}^n P^{m_t}\right)^{a_j}\otimes 
    \frac{1}{p^{a_jk}}J_{p^k}^{a_j}\right).\]
     Note $P^{m_t}=I_p$ for all $t$ expect $m=j$. Moving the coefficient $\zeta^{\alpha_j a_j}$ to be on the one non-identity term in the Kronecker product for each $j$, then in the product we have in each position of the Kronecker product $\zeta^{\alpha_j a_j}P_j$ multiplied by the identity repeatedly. This gives
    \[\frac{1}{p^{n+k}}\left(\left(\bigotimes_{t=1}^n I_p\right)\otimes J_{p^k}+p^k\sum_{a_1,a_2,\cdots, a_n}\zeta^{\alpha_j a_j}\left(\bigotimes_{j=1}^n P^{a_j}\right)\otimes \frac{1}{p^{k\sum_{i=1}^n a_i}}J_{p^k}^{\sum_{i=1}^n a_i}\right)\]
    \[=\frac{1}{p^{n+k}}\left(\left(\bigotimes_{t=1}^n I_p\right)\otimes J_{p^k}+p^k\sum_{a_1,a_2,\cdots, a_n}\zeta^{\alpha_j a_j}\left(\bigotimes_{j=1}^n P^{a_j}\right)\otimes \frac{1}{p^{k\sum_{i=1}^n a_i}}p^{k\sum_{i=1}^n a_i-1}J_{p^k}\right)\]
    \[=\frac{1}{p^{n+k}}\left(\left(\bigotimes_{t=1}^n I_p\right)\otimes J_{p^k}+\sum_{a_1,a_2,\cdots, a_n}\zeta^{\alpha_j a_j}\left(\bigotimes_{j=1}^n P^{a_j}\right)\otimes J_{p^k}\right).\]
    We recall the summation is over all $a_1,a_2\cdots, a_n\in \{0,1,\cdots, p-1\}$ such that not all of the $a_j$ are zero. Notice if $a_1=a_2=\cdots =a_n=0$, then we would have 
    \[\zeta^{\alpha_j a_j}\left(\bigotimes_{j=1}^n P^{a_j}\right)\otimes J_{p^k}=1\left(\bigotimes_{j=1}^n I_p\right)\otimes J_{p^k},\]
    Therefore, 
    \[\frac{1}{p^{n+k}}\left(\left(\bigotimes_{t=1}^n I_p\right)\otimes J_{p^k}+\sum_{a_1,a_2,\cdots, a_n}\zeta^{\alpha_j a_j}\left(\bigotimes_{j=1}^n P^{a_j}\right)\otimes J_{p^k}\right)=\frac{1}{p^{n+k}}\sum_{a_1,a_2,\cdots, a_n}\zeta^{\alpha_j a_j}\left(\bigotimes_{j=1}^n P^{a_j}\right)\otimes J_{p^k},\]
    where $a_1,a_2,\cdots,a_n\in \{0,1,\cdots p-1\}$. We next note that $J_{p^k}=\bigotimes_{j=1}^k J_p$. Notice that every row and column of $P$ has exactly one nonzero entry. Since every row and column of $P$ has the same number of nonzero entries, it commutes with $J_p$. As these matrices commute and are diagonalizable they are simultaneously diagonalizable. Let $Q$ be a matrix the simultaneously diagonalizes $P$ and $J_p$. Say that $QPQ^{-1}=D$ and $QJ_pQ^{-1}=K_p$. We note 
    \[P\sim\begin{pmatrix}
        1 & & & \\  & \zeta & & \\ & & & \ddots & \\ & & & & \zeta^{p-1} 
    \end{pmatrix}\text{ and } J_p\sim \begin{pmatrix}
        p & & & \\  & 0 & & \\ & & & \ddots & \\ & & & & 0 
    \end{pmatrix}.\]
    Then $D$ is a diagonal matrix with diagonal entries being a permutation of $1,\zeta,\zeta^2,\cdots, \zeta^{p-1}$ and $K_p$ is a diagonal matrix with one entry being $p$ and the rest being $0$. Picking $Q$ appropriately, we may assume that $K_p$ has the first diagonal entry as $p$. Now conjugating by $\bigotimes_{m=1}^{n+k} Q$ we have
    \[\frac{1}{p^{n+k}}\sum_{a_1,a_2,\cdots, a_n}\zeta^{\alpha_j a_j}\left(\bigotimes_{j=1}^n P^{a_j}\right)\otimes J_{p^k}\sim
    \bigotimes_{m=1}^{n+k} Q\left( \frac{1}{p^{n+k}}\sum_{a_1,a_2,\cdots, a_n}\zeta^{\alpha_j a_j}\left(\bigotimes_{j=1}^n P^{a_j}\right)\otimes J_{p^k}\right) \bigotimes_{m=1}^{n+k} Q^{-1}\]
    \[=\bigotimes_{m=1}^{n+k} Q \left(\frac{1}{p^{n+k}}\sum_{a_1,a_2,\cdots, a_n}\zeta^{\alpha_j a_j}\left(\bigotimes_{j=1}^n P^{a_j}\right)\otimes \bigotimes_{j=1}^k J_p\right) \bigotimes_{m=1}^{n+k} Q^{-1}\]
    \[=\frac{1}{p^{n+k}}\sum_{a_1,a_2,\cdots, a_n} \zeta^{\alpha_j a_j}\bigotimes_{j=1}^n QP^{a_j}Q^{-1}\otimes \bigotimes_{j=1}^k QJ_pQ^{-1}=\frac{1}{p^{n+k}}\sum_{a_1,a_2,\cdots, a_n} \zeta^{\alpha_j a_j}\bigotimes_{j=1}^n D^{a_j}\otimes \bigotimes_{j=1}^k K_p.\]
    We note that $K=\bigotimes_{j=1}^k K_p$ is a diagonal matrix with $p^k$ in the first position and $0$ elsewhere. Then we are considering
    \[\frac{1}{p^{n+k}}\sum_{a_1,a_2,\cdots, a_n} \zeta^{\alpha_j a_j}\bigotimes_{j=1}^n D^{a_j}\otimes \bigotimes_{j=1}^k K_p=\frac{1}{p^{n+k}}\sum_{a_1,a_2,\cdots, a_n} \zeta^{\alpha_j a_j}\bigotimes_{j=1}^n D^{a_j}\otimes K.\]
    We shall focus on 
    \[\frac{1}{p^{n+k}}\sum_{a_1,a_2,\cdots, a_n} \zeta^{\alpha_j a_j}\bigotimes_{j=1}^n D^{a_j}=\frac{1}{p^{n+k}}\sum_{a_1,a_2,\cdots a_n} \bigotimes_{j=1}^n \left(\zeta^{\alpha_j} D\right)^{a_j}.\] 
    By a similar argument to that in the proof of Theorem \ref{thm:2idm} this matrix
    is a diagonal matrix with $\frac{1}{p^k}$ in a single entry. Since $K$ is a matrix with $p^k$ in the first diagonal entry and $0$ in all other entries, 
    \[\frac{1}{p^{n+k}}\sum_{a_1,a_2,\cdots, a_n} \zeta^{\alpha_j a_j}\bigotimes_{j=1}^n D^{a_j}\otimes K\]
    has a single diagonal entry that is $1$ and all other entries are $0$. We then have
    \[\left(\frac{1}{p^{n+k}}\sum_{a_1,a_2,\cdots, a_n} \zeta^{\alpha_j a_j}\bigotimes_{j=1}^n D^{a_j}\otimes K\right)^2=\frac{1}{p^{n+k}}\sum_{a_1,a_2,\cdots, a_n} \zeta^{\alpha_j a_j}\bigotimes_{j=1}^n D^{a_j}\otimes K,\]
    so this matrix is an idempotent. Then as $Y_i'(\alpha)$ 
    is similar to this matrix, it must be an idempotent as well. Thus $Y_i(\alpha)$ is an idempotent since its only nonzero block is the $C_i,C_i$ block, which we just found to be an idempotent.    
    \end{proof}

    \begin{Prop}\label{prop:orth3'}
    Let $G$ have nilpotency class $3$. Suppose $|Z(G)|=p^k$, $k>0$, $|G'\colon Z(G)|=p^n$, $Z(G)=\langle z_1, z_2,\cdots, z_k\rangle$, and $G'/Z(G)=\langle g_1Z(G), g_2Z(G),\cdots, g_nZ(G)\rangle$. Let $C_i\subseteq G\setminus G'$ be a class. Let $C_0,C_1,\cdots, C_{p^k-1}$ be the central classes of $G$ and $C_{p^k-1+j}=g_jZ(G)$ for $1\leq j\leq n$. Let $\zeta$ be any primitive $p$th root of unity. Then using the notation from Theorem \ref{thm:3'idm} we have that if $\alpha\neq \beta$, then $Y_i(\alpha)$ is orthogonal to $Y_i(\beta)$.
\end{Prop}

\begin{proof}
    The proof of this result is similar to that of Proposition \ref{prop:orth2} in the sense that we first use the same argument as in Theorem \ref{thm:3'idm} to write the idempotents as Kronecker products. Then show the inner product is $0$ by much the same argument. 
\end{proof}

We next show the idempotents that we have constructed in Theorems \ref{thm:2idm}, \ref{thm:3idm}, \ref{thm:3'idm} are all in fact pairwise orthogonal to one another. 

 \begin{Prop}\label{prop:orth3''}
    Let $G$ have nilpotency class $3$. Suppose $|Z(G)|=p^k$,$k>0$, $|G'\colon Z(G)|=p^n$, $Z(G)=\langle z_1, z_2,\cdots, z_k\rangle$, and $G'/Z(G)=\langle g_1Z(G), g_2Z(G),\cdots, g_nZ(G)\rangle$. Let $C_i\subseteq G\setminus G'$ and $C_\ell\subseteq G'\setminus Z(G)$ be classes of $G$. Let $C_0,C_1,\cdots, C_{p^k-1}$ be the central classes of $G$ and $C_{p^k-1+j}=g_jZ(G)$ for $1\leq j\leq n$. Let $\zeta$ be any primitive $p$th root of unity. Let $W_\ell(\alpha)$ be any idempotent of $T(G)$ from Theorem \ref{thm:2idm} for any $\alpha=(\alpha_1,\alpha_2,\cdots, \alpha_k)$ with each of the $\alpha_j\in \{0,1,\cdots, p-1\}$. Let $X_i(\beta)$ be any idempotent of  $T(G)$ from Theorem \ref{thm:3idm} for any $\beta=(\beta_1,\beta_2,\cdots, \beta_k)$ with each of the $\beta_j\in \{0,1,\cdots, p-1\}$ not all $0$. Let $Y_i(\gamma)$ be any idempotent of $T(G)$ from Theorem \ref{thm:3'idm} for any $\gamma=(\gamma_1,\gamma_2,\cdots, \gamma_n)$ with each of the $\gamma_j\in \{0,1,\cdots, p-1\}$ not all $0$. Then $W_\ell(\alpha), X_i(\beta)$, and $Y_i(\gamma)$ are pairwise orthogonal. 
\end{Prop}

\begin{proof}
    Note that $W_\ell(\alpha)$ is only nonzero in the $C_\ell,C_\ell$ block and both $X_i(\beta),Y_i(\gamma)$ are only nonzero in the $C_i,C_i$ block. Since $C_i\neq C_\ell$ we have $W_\ell(\alpha)$ and $X_i(\beta)$ are nonzero in different blocks. Then $\langle W_\ell(\alpha),X_i(\beta)\rangle=\trace(W_\ell(\alpha)\overline{X_i(\beta)^T})=0$. Hence, $W_\ell(\alpha)$ and $X_i(\beta)$ are orthogonal. By similar reasoning $W_\ell(\alpha)$ and $Y_i(\gamma)$ are orthogonal. 

    Now we consider $X_i(\beta)$ and $Y_i(\gamma)$. We have
    \[X_i(\beta)=\frac{1}{p^k}\sum_{b_1,b_2,\cdots, b_k} \prod_{j=1}^k \zeta^{\beta_j b_j}(E_i^*A_jE_i^*)^{b_j}\]
    where $b_1,b_2,\cdots, b_k\in \{0,1,\cdots, p-1\}$ and
    \[Y_i(\gamma)=\frac{1}{p^{n+k}}\left(\sum_{j=0}^{p^k-1} E_i^*A_jE_i^*+\sum_{a_1,a_2,\cdots, a_n} p^k\prod_{j=1}^n \frac{\zeta^{\gamma_j a_j}}{p^{a_jk}}(E_i^*A_{p^k-1+j}E_i^*)^{a_j}\right)\]
    where $a_1,a_2,\cdots a_n\in \{0,1,\cdots, p-1\}$ and not all of the $a_j=0$. Since both $X_i(\beta)$ and $Y_i(\gamma)$ are $0$ outside of the $C_i,C_i$ block, the inner product of $X_i(\beta)$ and $Y_i(\gamma)$ is the same as if we only considered the inner product of the $C_i,C_i$ block of each. For this reason we let
    \[X_i'(\beta)=\frac{1}{p^k}\sum_{b_1,b_2,\cdots, b_k\in \{0,1,\cdots, p-1\}} \prod_{j=1}^k \zeta^{\beta_j b_j}(A_j')^{b_j}\]
    and
    \[Y_i'(\gamma)=\frac{1}{p^{n+k}}\left(\sum_{j=0}^{p^k-1} A_j'+\sum_{a_1,a_2,\cdots, a_n} p^k\prod_{j=1}^n \frac{\zeta^{\gamma_j a_j}}{p^{a_jk}}(A_{p^k-1+j}')^{a_j}\right)\]
    where $A_j'$ is the $C_i,C_i$ block of the matrix $A_j$. Let $P$ be the $p\times p$ matrix in (\ref{eq:tensor3}).
    Then using the same arguments as in the proofs of Theorems \ref{thm:3idm} and \ref{thm:3'idm} we have
    \[X_i'(\beta)=\frac{1}{p^k}\sum_{b_1,b_2,\cdots, b_k} \bigotimes_{i=1}^n I_p\otimes \zeta^{\beta_j b_j} \bigotimes_{j=1}^{k} P^{b_j} \text{ and } Y_i'(\gamma)=\frac{1}{p^{n+k}}\sum_{a_1,a_2,\cdots, a_n}\zeta^{\gamma_j a_j}\left(\bigotimes_{j=1}^n P^{a_j}\right)\otimes J_{p^k}.\]

     We next note that $J_{p^k}=\bigotimes_{j=1}^k J_p$ where $J_p$ is the $p\times p$ all $1$'s matrix. We also have $J_p=\sum_{m=0}^{p-1} P^m$. We note that $P$ is similar to the $p\times p$ matrix $D$ given in (\ref{eq:diag}).
    
    Let $Q'=\bigotimes_{m=1}^{n+k} Q$. Then we note $Q'X_i'(\beta)\overline{Y_i'(\gamma)^T}(Q')^{-1}$ is similar to $X_i'(\beta)\overline{Y_i'(\gamma)^T}$, so we have
    \[\langle X_i'(\beta),Y_i'(\gamma)\rangle=\trace(X_i'(\beta)\overline{Y_i'(\gamma)^T})=\trace(Q'X_i'(\beta)\overline{Y_i'(\gamma)^T}(Q')^{-1})=\trace(Q'X_i'(\beta)(Q')^{-1}Q'\overline{Y_i'(\gamma)^T}(Q')^{-1}).\]
     We note that
    \[Q'X_i'(\beta)(Q')^{-1}=\bigotimes_{m=1}^{n+k} Q \frac{1}{p^k}\sum_{b_1,b_2,\cdots, b_k} \bigotimes_{i=1}^n I_p\otimes \zeta^{\beta_j b_j} \bigotimes_{j=1}^{k} P^{b_j} \bigotimes_{m=1}^{n+k} Q^{-1}\]
    \[=\frac{1}{p^k}\sum_{b_1,b_2,\cdots, b_k} \bigotimes_{i=1}^n QI_pQ^{-1}\otimes \zeta^{\beta_j b_j} \bigotimes_{j=1}^{k} QP^{b_j}Q^{-1}=\frac{1}{p^k}\sum_{b_1,b_2,\cdots, b_k} \bigotimes_{i=1}^n I_p\otimes \zeta^{\beta_j b_j} \bigotimes_{j=1}^{k} D^{b_j}\]
    \[=\frac{1}{p^k}\sum_{b_1,b_2,\cdots, b_k} I_{p^n}\otimes\bigotimes_{j=1}^k \left(\zeta^{\beta_j}D\right)^{b_j}=I_{p^n}\otimes \frac{1}{p^k}\sum_{b_1,b_2,\cdots, b_k} \bigotimes_{j=1}^k \left(\zeta^{\beta_j}D\right)^{b_j}.\]
    We also have
    \[Q'Y_i'(\gamma)(Q')^{-1}=\bigotimes_{m=1}^{n+k} Q\left( \frac{1}{p^{n+k}}\sum_{a_1,a_2,\cdots, a_n}\zeta^{\gamma_j a_j}\left(\bigotimes_{j=1}^n P^{a_j}\right)\otimes J_{p^k}\right) \bigotimes_{m=1}^{n+k} Q^{-1}\]
    \[=\bigotimes_{m=1}^{n+k} Q \left(\frac{1}{p^{n+k}}\sum_{a_1,a_2,\cdots, a_n}\zeta^{\gamma_j a_j}\left(\bigotimes_{j=1}^n P^{a_j}\right)\otimes \bigotimes_{j=1}^k J_p\right) \bigotimes_{m=1}^{n+k} Q^{-1}\]
    \[=\frac{1}{p^{n+k}}\sum_{a_1,a_2,\cdots, a_n} \zeta^{\gamma_j a_j}\bigotimes_{j=1}^n QP^{a_j}Q^{-1}\otimes \bigotimes_{j=1}^k QJ_pQ^{-1}=\frac{1}{p^{n+k}}\sum_{a_1,a_2,\cdots, a_n} \zeta^{\gamma_j a_j}\bigotimes_{j=1}^n D^{a_j}\otimes \bigotimes_{j=1}^k K_p.\]
    Observe that $K=\bigotimes_{j=1}^k K_p$ is a diagonal matrix with $p^k$ in the first position and $0$ elsewhere. So
    \[Q'Y_i(\gamma)(Q')^{-1}=\frac{1}{p^{n+k}}\sum_{a_1,a_2,\cdots, a_n} \zeta^{\gamma_j a_j}\bigotimes_{j=1}^n D^{a_j}\otimes K.\]
    Hence,
    \[\langle X_i'(\beta),Y_i'(\gamma)=\trace\left(I_{p^n}\otimes \frac{1}{p^k}\sum_{b_1,b_2,\cdots, b_k} \bigotimes_{j=1}^k \left(\zeta^{\beta_j}D\right)^{b_j}\overline{\left(\frac{1}{p^{n+k}}\sum_{a_1,a_2,\cdots, a_n} \zeta^{\gamma_j a_j}\bigotimes_{j=1}^n D^{a_j}\otimes K \right)^T} \right)\]
    \[=\trace\left(\frac{1}{p^{n+k}}\overline{\left(\sum_{a_1,a_2,\cdots, a_n} \zeta^{\gamma_j a_j}\bigotimes_{j=1}^n D^{a_j}\right)^T}\otimes\left( \frac{1}{p^k}\sum_{b_1,b_2,\cdots, b_k} \bigotimes_{j=1}^k \left(\zeta^{\beta_j}D\right)^{b_j}\right)K\right).\]

    Note in $\bigotimes_{m=1}^k D$ that each entry is a product of powers of $\zeta$. With this observation we let $(t_1,t_2,\cdots, t_k)$ denote the entry of $\bigotimes_{m=1}^k D$ corresponding to $\zeta^{t_1}\zeta^{t_2}\cdots \zeta^{t_k}$. With this notation we saw in the proof of Theorem \ref{thm:3idm} that 
    \[\frac{1}{p^k}\sum_{b_1,b_2,\cdots, b_k\in \{0,1,\cdots, p-1\}} \zeta^{\beta_j b_j}\bigotimes_{j=1}^k D^{b_j}\]
    is a diagonal matrix with a single nonzero entry of $1$ in diagonal position $(-\beta_1,-\beta_2,\cdots, -\beta_k)$ (where each of these terms is taken modulo $p$). In particular, as not all of the $\beta_j$ are $0$ by assumption, the first diagonal entry of 
    \[B=\frac{1}{p^k}\sum_{b_1,b_2,\cdots, b_k} \zeta^{\beta_j b_j}\bigotimes_{j=1}^k D^{b_j}\]
    is $0$. We have that $B$ and $K$ are diagonal matrices where $B$ is $0$ in the first diagonal entry and $K$ is only nonzero in the first diagonal entry. Thus, $BK=0$. 
    Then
    \[\trace\left(\frac{1}{p^{n+k}}\overline{\left(\sum_{a_1,a_2,\cdots, a_n} \zeta^{\gamma_j a_j}\bigotimes_{j=1}^n D^{a_j}\right)^T}\otimes BK\right)=\trace\left(\frac{1}{p^{n+k}}\overline{\left(\sum_{a_1,a_2,\cdots, a_n} \zeta^{\gamma_j a_j}\bigotimes_{j=1}^n D^{a_j}\right)^T}\otimes0\right)=0.\]
    Therefore, $\langle X_i'(\beta),Y_i'(\gamma)\rangle=0$. Recall that this equals $\langle X_i(\beta),Y_i(\gamma)\rangle$. Hence, $X_i(\beta)$ and $Y_i(\gamma)$ are orthogonal.
\end{proof}

The idempotents $W_\ell(\alpha),X_i(\beta),Y_i(\gamma)$ are all of the non-primary primitive idempotents of $T(G)$. We show this by finding the Wedderburn decomposition.

\begin{Thm}\label{thm:cam3decomp}
    Let $G$ have nilpotency class $3$ with $|G:G'|=p^{2n}$, $|G':Z(G)|=p^n$ for some even integer $n$. Suppose $|Z(G)|=p^k$. Then the Wedderburn decomposition of $T(G)$ is
    \[T(G)=V\oplus \bigoplus_{i=1}^{p^{n}-1}\bigoplus_\alpha\mathcal{W}_i(\alpha)\oplus \bigoplus_{t=1}^{p^{2n}-1}\bigoplus_\beta \mathcal{X}_t(\beta)\oplus \bigoplus_{s=1}^{p^{2n}-1} \bigoplus_\gamma \mathcal{Y}_s(\gamma),\]
    where $V$ is the primary component, each $\mathcal{W}_i(\alpha)$ is a one-dimensional ideal generated by some $W_i(\alpha)$ where $\alpha=(\alpha_1,\alpha_2,\cdots, \alpha_k)$ with each $\alpha_j\in \{0,1,2,\cdots, p-1\}$ not all $\alpha_j$ are $0$. Each $\mathcal{X}_t(\beta)$ is a one-dimensional ideal generated by some $X_t(\beta)$ where $\beta=(\beta_1,\beta_2,\cdots, \beta_k)$ with each $\beta_j\in \{0,1,\cdots, p-1\}$ not all $\beta_j$ are zero. Each $\mathcal{Y}_s(\gamma)$ is a one-dimensional ideal generated by some $Y_s(\gamma)$ where $\gamma=(\gamma_1,\gamma_2,\cdots, \gamma_n)$ with each $\gamma_j\in \{0,1,\cdots, p-1\}$ not all $\gamma_j$ are zero. 
\end{Thm}

\begin{proof}
    By Proposition \ref{cor:p3dim}, $\dim T(G)=3p^{k+n}+3p^{k+2n}+p^{2k}-6p^k+p^{4n}+3p^{3n}-5p^{2n}-6p^n+7$. We know $V$ is an irreducible ideal with dimension equal to the number of classes of $G$ squared. From the proof of Proposition \ref{cor:p3dim}, $\dim(V)=(p^{2n}+p^n+p^k-2)^2$. We can then write $T(G)$ as $T(G)=V\oplus R$ where $R$ is an ideal with $\dim (R)=p^{k+n}+p^{k+2n}-2p^k+p^{3n}-2p^{2n}-2p^n+3=(p^n-1)(p^k-1)+(p^{2n}-1)(p^k-1)+(p^{2n}-1)(p^n-1)$. 

    By the same argument as in the proof of Theorem \ref{thm:cam2decomp} we have for all $C_i\in G'\setminus Z(G)$ and $\beta\neq (0,0,\cdots, 0)$ that $W_i(\beta)\in R$.  Now let $\mathcal{W}_{i}(\alpha)$ be the ideal generated by $W_i(\alpha)$. Then for all $\alpha\neq (0,0,\cdots, 0)$, $\mathcal{W}_{i}(\alpha)\subseteq R$. There are a total of $p^k-1$ choices for $\alpha\neq (0,0,\cdots, 0)$. As $|G'\colon Z(G)|=p^n$ and every class $C_i\subseteq G'\setminus Z(G)$ is of the form $xZ(G)$ by Proposition \ref{lem:pcamclass}, there are $p^n-1$ choices for $C_i$. Hence, we have found $(p^k-1)(p^{n}-1)$ orthogonal ideals all in $R$.
    
    From Proposition \ref{prop:orth3'} we have for each choice of $C_\ell\subseteq G\setminus G'$ idempotents $Y_\ell(\alpha)$ all of which are orthogonal. Clearly for $\ell\neq j$ we have that $Y_\ell(\alpha)$ and $Y_j(\beta)$ are orthogonal as they are nonzero in different blocks. For any fixed $C_\ell\subseteq G\setminus G'$, we consider $Y_\ell(\alpha)$ when $\alpha=(0,0,\cdots, 0)$. We have
    \[Y_\ell(\alpha)=\frac{1}{p^{n+k}}\left(\sum_{j=0}^{p^k-1} E_\ell^*A_jE_\ell^*+\sum_{a_1,a_2,\cdots, a_n} p^k\prod_{j=1}^n \frac{1}{p^{a_jk}}(E_\ell^*A_{p^k-1+j}E_\ell^*)^{a_j}\right).\]
    where not all of the $a_i$ are $0$.
    
    By Lemma \ref{lem:nil3structure2'}, $p^k\prod_{j=1}^n\frac{1}{p^a_j k} (E_\ell^*A_jE_\ell^*)^{a_j}=E_\ell^*A_{u}E_\ell^*$ where $C_u=g_1^{a_1}g_2^{a_2}\cdots g_n^{a_n}Z(G)$. Then this means that
    \[\sum_{a_1,a_2,\cdots, a_n} p^k\prod_{j=1}^n \frac{1}{p^{a_jk}}(E_\ell^*A_{p^k-1+j}E_\ell^*)^{a_j}\]
    is just a sum of the $C_\ell,C_\ell$ blocks of all the adjacency matrices corresponding to a conjugacy class in $G'\setminus Z(G)$. We note that $\sum_{j=0}^{p^k-1} E_\ell^*A_jE_\ell^*$
    is just a sum of the $C_\ell,C_\ell$ blocks of all the adjacency matrices corresponding to the class $C_\ell$ in $Z(G)$. Then $Y_\ell(\alpha)$ is the sum of the $C_\ell,C_\ell$ block of all the conjugacy classes contained in $G'$. 
    
    Let $C_t$ be any class of $G$ outside of $G'$. Say $C_\ell=xG'$ and $C_t=yG'$. Then $C_\ell C_t=xyG'$. If $x\in xyG'$, then $x=xyz$ for some $z\in G'$ and then $y^{-1}=z$. This implies $y^{-1}\in G'$, which is false. Therefore, $C_\ell\not\subseteq C_\ell C_t$. Then by Lemma \ref{lem:cam2outer} the $C_\ell,C_\ell$ block of $A_t$ is all $0$. Hence, every adjacency matrix corresponding to a class outside of $G'$ is $0$ in the $C_\ell,C_\ell$ block. As the sum of all the adjacency matrices is the all $1$'s matrix, the sum of all the adjacency matrices corresponding to $C_t$ outside $G'$ must be $0$ in the $C_\ell,C_\ell$ block. Hence, the sum of all the adjacency matrices corresponding to $C_t\subseteq G'$ must be the all $1$'s matrix in the $C_\ell,C_\ell$ block. Hence, we have that all the entries of $Y_\ell(\alpha)$ in the $C_\ell,C_\ell$ block are $1$ and all the entries outside of this block are $0$. Therefore, it is equal to the basis element of $V$ that is the all $1$'s matrix in block $C_\ell,C_\ell$ and $0$ outside this block. We denote this as $v_{\ell\ell}$. 

    Then for all $\gamma\neq (0,0,\cdots, 0)$ we have $\langle Y_\ell(\alpha),Y_\ell(\gamma)\rangle=0$, so $\langle v_{\ell\ell},Y_\ell(\gamma)\rangle=0$ as $Y_\ell(\alpha)$ is a multiple of $v_{\ell\ell}$. Clearly we have $\langle v_{st},X_\ell(\gamma)\rangle=0$ for all $s,t$ not both equal to $\ell$. Hence, for all $\gamma\neq (0,0,\cdots, 0)$ we have $Y_\ell(\gamma)$ is orthogonal to $V$. Then $Y_\ell(\gamma)\in R$. As we chose $C_\ell$ arbitrarily, this is true for all $C_\ell\subseteq G\setminus G'$.

    Now let $\mathcal{Y}_{\ell}(\gamma)$ be the ideal generated by $Y_\ell(\gamma)$. Then for all $\gamma\neq (0,0,\cdots, 0)$, $\mathcal{Y}_{\ell}(\gamma)\subseteq R$. There are a total of $p^n-1$ choices for $\gamma\neq (0,0,\cdots, 0)$. As $|G\colon G'|=p^{2n-1}$ an every class $C_\ell\subseteq G\setminus G'$ is of the form $xG'$ by Proposition \ref{lem:pcamclass}, there are $p^{2n}-1$ choices for $C_\ell$. Hence, we have found $(p^n-1)(p^{2n}-1)$ orthogonal ideals all in $R$. By Proposition \ref{prop:orth3''}, each $Y_\ell(\alpha)$ is orthogonal to each $W_i(\beta)$. Then none of the $\mathcal{W}_{i}(\alpha)$ and $\mathcal{Y}_{\ell}(\gamma)$ can intersect. We hence have found $(p^k-1)(p^n-1)+(p^n-1)(p^{2n}-1)$ orthogonal ideals all in $R$.

    By Proposition \ref{prop:orth3v}, each of the $X_\ell(\alpha)$ with not all the $\alpha_j$ are zero, is orthogonal to $V$. Then $X_\ell(\beta)\in R$ for all $\ell$ and $\beta$ with $\beta_j$ not all $0$.  Now let $\mathcal{X}_{\ell}(\beta)$ be the ideal generated by $X_\ell(\beta)$. Then for all $\beta\neq (0,0,\cdots, 0)$, $\mathcal{X}_{\ell}(\beta)\subseteq R$. There are a total of $p^k-1$ choices for $\beta\neq (0,0,\cdots, 0)$. As $|G\colon G'|=p^{2n}$ and every class $C_\ell\subseteq G\setminus G'$ is of the form $xG'$ by Proposition \ref{lem:pcamclass}, there are $p^{2n}-1$ choices for $C_\ell$. Hence, we have found $(p^k-1)(p^{2n}-1)$ orthogonal ideals all in $R$. By Proposition \ref{prop:orth3''}, each of the $X_\ell(\beta)$ are orthogonal to each of the $W_i(\alpha)$ and each of the $Y_\ell(\gamma)$. Then none of the $\mathcal{W}_{i}(\alpha)$, $\mathcal{X}_{\ell}(\beta)$, and $\mathcal{Y}_{\ell}(\gamma)$ can intersect. We hence have found $(p^k-1)(p^n-1)+(p^k-1)(p^{2n}-1)+(p^{2n}-1)(p^n-1)$ orthogonal ideals all in $R$. This equals $\dim(R)$, so each $\mathcal{W}_{i}(\alpha)$, $\mathcal{X}_{\ell}(\beta)$, $\mathcal{Y}_{\ell}(\gamma)$ must be $1-$dimensional and thus irreducible. Decomposing $R$ into its irreducible components we have
    \[T(G)=V\oplus \bigoplus_{i=1}^{p^{n}-1}\bigoplus_\alpha\mathcal{W}_i(\alpha)\oplus \bigoplus_{t=1}^{p^{2n}-1}\bigoplus_\beta \mathcal{X}_t(\beta)\oplus \bigoplus_{s=1}^{p^{2n}-1} \bigoplus_\gamma \mathcal{Y}_s(\gamma).\qedhere\]
\end{proof}

\bibliographystyle{plain}
\bibliography{Sources}

\end{document}